\documentclass[12pt]{article}

\usepackage{amsmath,amssymb,amsthm}
\usepackage{latexsym}
\usepackage{graphicx}
\usepackage{psfrag}

\newcommand*{\E}{\mathrm{\mathbf{E}}}
\renewcommand{\theenumi}{\roman{enumi}}
\renewcommand{\labelenumi}{{\rm(\theenumi)}}

\newtheorem{theorem}{Theorem}[section]
\newtheorem{lemma}[theorem]{Lemma}
\newtheorem{corollary}[theorem]{Corollary}
\newtheorem{proposition}[theorem]{Proposition}

\def\nfrac#1#2{{\textstyle\frac{#1}{#2}}}
\def\dfrac#1#2{\lower0.15ex\hbox{\large$\frac{#1}{#2}$}}

\newcommand{\C}{\mathcal{W}}
\newcommand{\W}{\mathcal{F}}
\newcommand{\dvec}{\boldsymbol{d}}

\begin{document}

\makeatletter
\renewcommand{\theenumi}{\roman{enumi}}
\renewcommand{\labelenumi}{(\theenumi)}
\renewcommand{\theenumii}{\alph{enumii}}
\renewcommand{\labelenumii}{(\theenumii)}
\renewcommand{\p@enumii}{}
\renewcommand{\theenumiii}{\arabic{enumii}}
\renewcommand{\labelenumiii}{(\theenumiii)}
\renewcommand{\p@enumiii}{}
\makeatother

\setlength{\abovedisplayskip}{0.5\baselineskip}
\setlength{\belowdisplayskip}{0.5\baselineskip}
\setlength{\abovedisplayshortskip}{0.25\baselineskip}
\setlength{\belowdisplayshortskip}{0.25\baselineskip}
\setlength{\itemsep}{0pt}

\title{A polynomial bound on the mixing time of a Markov chain for 
sampling regular directed graphs}

\author{Catherine Greenhill}

\author{
Catherine Greenhill\\
\small School of Mathematics and Statistics\\[-0.5ex]
\small The University of New South Wales\\[-0.5ex]
\small Sydney NSW 2052, Australia\\[-0.5ex]
\small \tt csg@unsw.edu.au}

\date{}

\maketitle

\begin{abstract}
The switch chain is a well-known Markov chain for sampling
directed graphs with a given degree sequence.  
While not ergodic in general, we show that it is ergodic for
regular degree sequences.  We then prove that the switch chain is
rapidly mixing for regular directed graphs of degree $d$, where 
$d$ is any positive integer-valued function of the number of vertices.
We bound the mixing time by bounding the eigenvalues of the chain.
A new result is presented and applied to bound the smallest 
(most negative) eigenvalue.  This result is a modification
of a lemma by Diaconis and Stroock~\cite{DS91}, and by using it
we avoid working with a lazy chain.  A multicommodity flow argument is 
used to bound the second-largest eigenvalue of the chain. 
This argument is based on the
analysis of a related Markov chain for undirected regular graphs
by Cooper, Dyer and Greenhill~\cite{CDG}, but with significant
extension required.  
\end{abstract}

\section{Introduction}\label{s:intro}

Directed graphs are natural combinatorial objects which are used to model
systems in many areas including
biology (for example~\cite{protein,LC}), the social sciences (for 
example~\cite{RPW,JR})
and computer science (for example~\cite{HC,MS}).
In this paper we consider the problem of sampling directed graphs
with a given degree sequence.

For graph-theoretic terminology not introduced here, see~\cite{BJG}.
A directed graph (digraph) $G=(V,A)$ consists of a vertex set $V=V(G)$
and an arc set 
\[ A=A(G)\subseteq \{ (v,w)\in V\times V\mid v\neq w\}.\]
Note that digraphs as defined here are simple, which means that they
contain no loops and no multiple arcs.

The arc $(v,w)$ is drawn as an arrow from $v$ to $w$.  We refer
to $v$ as the \emph{tail} and $w$ as the \emph{head} of the arc.
For a vertex $v$, the \emph{out-degree} $d^+(v)$ of $v$ is the number of
arcs with tail $v$.  Similarly, the \emph{in-degree} $d^-(v)$ of $v$
is the number of arcs with head $v$.  For a positive integer $d$,
if $d^+(v)=d^-(v)=d$
for all vertices $v\in V$ then we say that the digraph $G$
        is \emph{$d$-regular} (or \emph{$d$-in, $d$-out}).

Let $d = d(n)\geq 1$ be a sequence of positive integers,
and let $\Omega_{n,d}$ be the set of all simple $d$-regular
digraphs on the vertex set $[n]=\{ 1,\ldots, n\}$.
The \emph{configuration model} of Bollob{\' a}s~\cite{bollobas}
(adapted for directed graphs)
gives an expected polynomial-time uniform sampling algorithm
for $\Omega_{n,d}$ when $d=O(\sqrt{\log n})$.  

There is a one-to-one correspondence between $\Omega_{n,d}$
and the set of all $d$-regular bipartite graphs on 
$\{ 1,2,\ldots, n\}\cup\{n+1,n+2,\ldots, 2n\}$ with no edges
in common with the perfect matching $\{ \{ j,n+j\} : j=1,\ldots, n\}$.
The probability that a $d$-regular bipartite graph on the given vertex 
bipartition has no edges in common with this perfect matching is
asymptotic to $e^{-d}$ whenever $d=o(n^{1/3})$,
by ~\cite[Theorem 4.6]{McKay84}.  This probability is 
polynomially small when $d=O(\log n)$.  McKay and Wormald's
algorithm~\cite{MW90} for sampling $d$-regular 
graphs runs in expected polynomial time for $d = O(n^{1/3})$,
and hence gives rise to an expected polynomial-time algorithm
for uniformly sampling elements of $\Omega_{n,d}$ when $d=O(\log n)$.

The set of all 1-regular digraphs is in one-to-one correspondence
with the set of all derangements of $[n]$,
and here the configuration model corresponds to repeatedly
sampling uniform
permutations of $n$ until one is obtained without fixed points.
The proportion of permutations which are derangements tends to $1/e$,
so this algorithm has linear expected running time.
Other algorithms for uniformly sampling derangements in 
linear expected time but an improved
constant have been proposed, for example~\cite{MPP}.

We know of no expected polynomial-time
uniform sampling algorithm for regular digraphs other than those
mentioned above.
Hence we turn our attention to the problem of obtaining approximately uniform 
samples from $\Omega_{n,d}$  using a Markov chain.  
(Some Markov chain definitions are given in Section~\ref{s:intro-MC};
for others, see~\cite{OH}.)

There is a very natural Markov chain for digraphs which has arisen in
many contexts,  which we will call the \emph{switch chain}.
A transition of the switch chain is performed by randomly choosing
two distinct arcs and exchanging their heads, if the two arcs are
non-incident and if the resulting digraph does not
contain any multiple arcs.
See Figure~\ref{f:chain} for a precise description of 
the transition procedure of the chain.
A transition of the switch chain is called 
\emph{switching along an alternating rectangle}
by Rao, Jana and Bandyopadhyay~\cite{RJB};  we will
simply call it a \emph{switch}.  Similar transformations
were used by Ryser~\cite{ryser} to study 0-1 matrices.  
Besag and Clifford~\cite{BC} defined a related chain for sampling 
0-1 matrices with given row and column sums, while
Diaconis and Sturmfels~\cite{DS} used a similar chain to sample 
contingency tables. 

Rao, Jana and Bandyopadhyay~\cite{RJB} showed that the switch chain
is not irreducible for general degree sequences.  (However, they
mention that degree sequences for which the switch chain is
not irreducible are ``rather rare''.)
For completeness, we prove in Lemma~\ref{connected} that
the switch chain is irreducible for regular digraphs. 
This also follows from the existence of the multicommodity flow
defined in Sections~\ref{s:flow},~\ref{s:2circuit}.

The switch chain is aperiodic for $d\geq 1$, 
as we prove in Lemma~\ref{aperiodic}.

In their empirical study of methods for generating directed graphs
with given degree sequences, Milo et al.~\cite{MKINA} wrote
that the switch chain ``works well but, as with many Markov chain methods, 
suffers because
in general we have no measure of how long we need to wait for it to
mix properly''.   
Our main result, Theorem~\ref{main}, 
partially answers this point by providing 
the first rigorous
polynomial bound on the mixing time of the switch chain, in the
special case of regular digraphs.  
\begin{theorem}
Let $\Omega_{n,d}$ be the set of all $d$-regular digraphs on
the vertex set $\{ 1,\ldots, n\}$, where $d=d(n)$ is any integer-valued
function which satisfies $1\leq d(n)\leq n-1$ for all $n\geq 4$.
Let $\tau(\epsilon)$ be the mixing time of the Markov chain
$\mathcal{M}$ with state space $\Omega_{n,d}$ and transition procedure
given by Figure~\ref{f:chain}, for $d\geq 1$. 
Then
\[ \tau(\epsilon) \leq 50\, d^{25}\, n^{9}\,
           \left( dn\log(dn) + \log(\epsilon^{-1})\right).
	   \]
\label{main}
\end{theorem}
Our proof of this result has two parts.  To avoid using a lazy chain
(which stays where it is at each step, with probability at least 
$\nfrac{1}{2}$)
we prove and apply a new result which can be used to bound the
smallest eigenvalue of an ergodic reversible Markov chain. 
This new bound is based on
Diaconis and Stroock~\cite[Proposition 2]{DS91} and inspired by
Sinclair~\cite[Theorem 5]{sinclair}.  To bound the second-largest
eigenvalue of the chain we adapt the multicommodity
flow analysis given in~\cite{CDG} for the undirected case. 
While some parts of the proof are very similar to~\cite{CDG},
significant extra technical difficulties arise
in the directed setting.  
We expect that the bound on the mixing time given in Theorem~\ref{main}
is far from tight,  but proving a substantially tighter bound 
seems beyond the reach of known proof techniques.

The \emph{flip chain} is a Markov chain which performs a restricted set
of switches, designed to ensure that the underlying digraph never becomes
disconnected.  
The flip chain for undirected graphs was described in~\cite{MS1},
and proposed as a self-randomizing mechanism for peer-to-peer networks.
The mixing time of the flip chain for regular undirected graphs
was analysed in~\cite{CDH,FGMS}, building on the multicommodity flow
analysis of the switch chain~\cite{CDG}.
We expect that Theorem~\ref{main} can be used to show that the flip chain
for digraphs is rapidly mixing for regular degree sequences.
This result would be of interest since
many protocols for communications networks (such as peer-to-peer networks)
use directed communications (see for example~\cite{HC,MS}).

The structure of the rest of paper is as follows.  
The necessary Markov chain definitions are given in the next subsection,
together with the new result (Lemma~\ref{smallest-eigval})
for bounding the smallest eigenvalue of an ergodic, reversible Markov chain.  
In Section~\ref{s:markov} we define the switch chain $\mathcal{M}$ and
prove that it is ergodic on $\Omega_{n,d}$ for $d\geq 1$.
A bound on the smallest eigenvalue of the chain is given in
Lemma~\ref{our-smallest-eigval}, and a bound on the second-largest
eigenvalue is stated in 
Proposition~\ref{our-second-largest-eigval}.
To conclude Section~\ref{s:markov}, we show how Theorem~\ref{main}
follows from Proposition~\ref{our-second-largest-eigval}, and
give an overview of the main steps of the multicommodity flow
argument which is used to prove 
Proposition~\ref{our-second-largest-eigval}.   This argument is
presented in Sections~\ref{s:flow}--~\ref{s:analysis}.
Finally, a worked example is given in Section~\ref{a:example}
which illustrates several features of the multicommodity flow
construction.

Before we begin our analysis, we mention some recent related work.
In many practical situations, almost uniformly random samples are required
in order to estimate the average value of some observable of
a system.  Kim et al.~\cite{KDBT} describe an alternative approach
to this problem in the case of sampling directed graphs with given
in-degrees and out-degrees.  Let $\dvec^+$, $\dvec^-$ be two vectors of
nonnegative integers with a common sum.   Denote by 
$\Omega_{n,\dvec^+,\dvec^-}$ the set of all digraphs on the vertex set 
$[n]$ with in-degree sequence $\dvec^+$ and 
out-degree sequence $\dvec^-$ (and assume that this set is nonempty). 
Kim et al.\ describe
an algorithm which runs in time $O(n^3)$ and produces a
random element of $\Omega_{n,\dvec^+,\dvec^-}$, drawn from
a specific non-uniform distribution. 
The samples output by the algorithm are statistically independent,
and the algorithm can calculate the weight of each digraph that
it produces.  They then explain how combining their algorithm
with biased sampling allows the average value of any function
on $\Omega_{n,\dvec^+,\dvec^-}$ to be approximated.  However, they
do not analyse the running time of the biased sampling procedure,
which could be very inefficient when the output distribution is far
from uniform.  (Indeed, in~\cite[Section 4.1]{KDBT} they assume
that the number of samples in the biased sampling is some positive
integer multiple of $|\Omega_{n,\dvec^+,\dvec^-}|$, which is usually
exponentially large.)

We complete this section with a final remark.
Milo et al.~\cite{MKINA} wrote of the
switch chain for directed graphs that ``Theoretical bounds on the mixing 
time exist only for specific near-regular degree sequences'', citing
Kannan, Tetali and Vempala~\cite{KTV}.
However, this is not correct, as we now explain.  
Two Markov chains are considered in~\cite{KTV}.
The first is an analogue of the switch chain for undirected graphs.
A bound on the mixing time is given in~\cite{KTV} for near-regular bipartite
undirected graphs, but no conclusion can be drawn from this for directed
graphs.  The second chain analysed in~\cite{KTV} is a Markov chain for
tournaments with a given score sequence.  
(A \emph{tournament} is a digraph obtained by giving an orientation to
each edge in an (undirected) complete graph.  Its \emph{score sequence}
is the sequence of out-degrees.)   Each transition of the Markov chain
reverses the arcs of a directed 3-cycle, so it is quite different from the
switch chain.  Furthermore, tournaments are very special kinds of digraphs.
We  know of no rigorous
polynomial bound on the mixing time of the switch chain for digraphs,
other than Theorem~\ref{main}.

\medskip

\noindent \emph{Acknowledgements.}\ 
I am grateful to Brendan McKay for his suggestion that 
it seemed unnecessary to make the switch chain lazy, which led
to the approach taken here.
I am also grateful to the anonymous referee for their helpful comments,
which improved both the content and the structure of this paper.

\subsection{Markov chain definitions and a new bound on the smallest eigenvalue}
\label{s:intro-MC}

Let $\mathcal{M}$ be an ergodic, time-reversible
Markov chain on the finite state space $\Omega$
with transition matrix $P$ and stationary distribution $\pi$.
The \emph{total variation distance}
between two probability distributions $\sigma,\,\sigma'$ on $\Omega$ is
given by
\[ d_{\mathrm{TV}}(\sigma,\sigma') = \tfrac{1}{2} \sum_{x\in \Omega}
          |\sigma(x) - \sigma'(x)|.\]
The \emph{mixing time} $\tau(\varepsilon)$ is defined by
\[ \tau(\varepsilon) = \mathrm{max}_{x\in \Omega}\,
          \mathrm{min}\left\{T\geq 0 \mid  d_{\mathrm{TV}}(P^t_x,\pi)
        \leq \varepsilon \mbox{ for all } t\geq T\right\},\]
where $P_x^t$ is the distribution  of the state $X_t$ of the Markov
chain after $t$ steps from the initial state $X_0=x$.
Let $\pi^\ast = \min\{ \pi(x) \mid x\in\Omega\}$ be the minimum
stationary probability.

The transition matrix $P$ has real eigenvalues
\[ 1 = \lambda_0 > \lambda_1 \geq \lambda_2 \geq \cdots
          \geq \lambda_{N-1} \geq -1,\]
where $N=|\Omega|$, and the Markov chain is aperiodic if
and only if $\lambda_{N-1} > -1$.  Let
\begin{equation}
\label{eigenvalues}
 \lambda_*= \max\{ \lambda_1, \, |\lambda_{N-1}|\}
\end{equation}
be the second-largest eigenvalue in absolute value.
The following result follows from Sinclair~\cite[Proposition 1]{sinclair},
which is based on a result of
Diaconis and Stroock~\cite[Proposition 3]{DS91}.

\begin{lemma} \emph{(\cite[Proposition 1]{sinclair})}
The mixing time of the Markov chain $\mathcal{M}$ satisfies
\[ \tau(\varepsilon) \leq (1-\lambda_*)^{-1}
                \left( \log(1/\pi^\ast) + \log(\varepsilon^{-1})\right).\]
\label{mix-eigvals}
\end{lemma}

It has become common practice when applying this bound to first
make the Markov chain lazy (that is, replace the transition matrix $P$
by $(I+P)/2$).  This ensures that all eigenvalues of the chain are nonnegative,
so that $\lambda_* = \lambda_1$ and only $(1-\lambda_1)^{-1}$ needs to be bounded.
However, we prefer not to use introduce unnecessary laziness and seek
an alternative approach. 

Diaconis and Stroock proved a result~\cite[Proposition 2]{DS91} which
provides an upper bound on $(1 + \lambda_{N-1})^{-1}$, where
$\lambda_{N-1}$ is the smallest eigenvalue of a Markov chain
as in (\ref{eigenvalues}).  In Lemma~\ref{smallest-eigval} below,
we give a new method for bounding on $\lambda_{N-1}$. 
The new bound is obtained by modifying~\cite[Proposition 2]{DS91}
in the same way that Sinclair modified~\cite[Proposition 1]{DS91}
to produce~\cite[Theorem 5]{sinclair}.
The modification results in a bound which is more local in character 
and seems easier to apply than~\cite[Proposition 2]{DS91}.
(See also the discussion in~\cite[Section 2]{sinclair}.)

To state the new bound we need some notation.
Write $\mathcal{G}$ for the underlying graph of the Markov chain
$\mathcal{M}$, where $\mathcal{G} = (\Omega, \Gamma)$
and each edge $e\in \Gamma$ corresponds to a transition of
$\mathcal{M}$.  That is, $e=\{ x,y\}$ is an edge of $\mathcal{G}$
if and only if $P(x,y)>0$.
Define $Q(e) = Q(x,y) = \pi(x)P(x,y)$ for the edge $e=\{x,y\}$.  
(If $P(x,x)>0$ then the edge $\{x,x\}$ is called a \emph{self-loop}
at $x$.)

For each $x\in\Omega$, fix a particular cycle 
from $x$ to $x$ in $\mathcal{G}$ with an odd number
of edges, and denote it by $\sigma_x$. 
(Such a cycle exists for each $x$, since the Markov chain
is aperiodic.)  Note that $\sigma_x$ may be a 1-cycle, which is
a walk along a self-loop edge at $x$.  Write $|\sigma_x|$ to denote
the length of the cycle $\sigma_x$, which is a positive odd number.
Let $\Sigma = \{ \sigma_x : x\in \Omega\}$ be the set of these odd cycles, 
and define the parameter
\[ \eta = \eta(\Sigma) = \operatorname{max}_{e\in\Gamma} \, \frac{1}{Q(e)}
\, \sum_{x\in\Omega,\,\, e\in\sigma_x} |\sigma_x| \pi(x).
\]

\begin{lemma} 
Suppose that $\mathcal{M}$ is a reversible, ergodic Markov chain with
state space $\Omega$.  Let $N=|\Omega|$ and let the 
eigenvalues of $\mathcal{M}$ be given by \emph{(\ref{eigenvalues})}.
Then
\[ (1 + \lambda_{N-1})^{-1} \leq  \frac{\eta}{2}.\]
\label{smallest-eigval}
\end{lemma}

\begin{proof}
The proof is very similar to the proof of~\cite[Proposition 2]{DS91},
but using a different application of the Cauchy-Schwarz inequality,
as in the proof of~\cite[Theorem 5]{sinclair}.  Assign an arbitrary
orientation to each cycle $\sigma_x$ and let $e = (e^-,e^+)$ under
this orientation.  Also define $\ell(e)$ to be the distance from $x$ to $e^-$
along the oriented cycle $\sigma_x$.  Then for any function
$\psi:\Omega\to\mathbb{R}$ we have
\[ \psi(x) = \dfrac{1}{2} \sum_{e\in\sigma_x} (-1)^{\ell(e)}\, 
           (\psi(e^+) + \psi(e^-))\]
for all $x\in\Omega$.
Given $\psi,\varphi:\Omega\to\mathbb{R}$, let
\[ 
\langle \psi, \varphi\rangle_\pi = \sum_{x\in\Omega} \psi(x)\varphi(x)\pi(x),\quad
\E_\pi(\psi) =  \sum_{x\in\Omega} \psi(x)\pi(x).
\]
Then for any nonzero function $\psi:\Omega\to\mathbb{R}$ we have
\begin{align*}
\E_\pi(\psi^2) = \sum_{x\in\Omega} \psi(x)^2\pi(x)
           &= \sum_{x\in\Omega} \pi(x)\,
 \left(\dfrac{1}{2}\, \sum_{e\in \sigma_x} 
              (-1)^{\ell(e)} (\psi(e^+) + \psi(e^-)) \right)^2\\
           &\leq \dfrac{1}{4} \sum_x \pi(x) |\sigma_x| \sum_{e\in\sigma_x}
                          (\psi(e^+) + \psi(e^-))^2,
                          \end{align*}
using the Cauchy-Schwarz inequality.  Exchanging the order of summation 
(and now orienting each edge $e\in\Gamma$ arbitrarily) gives
\begin{align*}
\E_\pi(\psi^2) &= \dfrac{1}{4} \sum_{e\in\Gamma} (\psi(e^+) + \psi(e^-))^2\,
                        \sum_{x\in\Omega,\, e\in\sigma_x}
                               |\sigma_x|\pi(x)\\
            &\leq \frac{\eta}{4}\, \sum_{e\in\Gamma} (\psi(e^+) + \psi(e^-))^2\, Q(e)\\
            &= \frac{\eta}{2}\, 
               \left(\E_\pi(\psi^2) + \langle \psi, P\psi\rangle_\pi\right).
\end{align*}
Divide through by $\E_\pi(\psi^2)$ to obtain
\[ 1\leq \frac{\eta}{2}\, \left(1 + \frac{\langle \psi, P\psi\rangle_\pi}
                                             {\E_\pi(\psi^2)}\right).\]
Now set $\psi$ equal to any eigenfunction $\psi_{N-1}$
corresponding to $\lambda_{N-1}$.  After rearranging this completes the proof,
since
\[ \langle \psi_{N-1}, P\psi_{N-1}\rangle_\pi  = 
   \lambda_{N-1}\, \langle \psi_{N-1}, \psi_{N-1}\rangle_\pi = 
             \lambda_{N-1}\, \E_\pi(\psi_{N-1}^2).\]
\end{proof}

This leads to an analogue of~\cite[Corollary 6]{sinclair}. 
We also prove a bound for a special case which often arises.

\begin{corollary}
\label{congestion-corollary}
Under the same conditions as Lemma~\emph{\ref{smallest-eigval}}\ we have
\[ (1 + \lambda_{N-1})^{-1} \leq \frac{\eta'(\Sigma) \, \ell(\Sigma) }{2}
\]
where
\[ \eta'(\Sigma) = \operatorname{max}_{e\in \Gamma} \frac{1}{Q(e)}
                \sum_{x\in\Omega,\, e\in\sigma_x} \pi(x), \qquad
  \ell(\Sigma) = \operatorname{max}_{x \in \Omega} |\sigma_x|.
  \]
In particular, if $\ell(\Sigma) = 1$ then
\[ (1+\lambda_{N-1})^{-1} \leq \dfrac{1}{2}\operatorname{max}_{x\in\Omega} 
P(x,x)^{-1}
\]
(where $P$ denotes the transition matrix of the Markov chain).
\end{corollary}

\begin{proof}
The first statement follows immediately from Theorem~\ref{smallest-eigval}.
Now suppose that $\ell(\Sigma)=1$.   Then each 
$\sigma_x\in\Sigma$ is a self-loop.
If $e=(y,y)\in\Gamma$ is a self-loop at $y$ then $e$ is contained
in exactly one element of $\Sigma$, namely $\sigma_y$. 
In this case
\[ \frac{1}{Q(e)}\, \sum_{x,\,e\in \sigma_x}\, \pi(x) = 
   \frac{\pi(y)}{Q(y,y)} = P(y,y)^{-1}. \]
If $e\in\Gamma$ is not a self-loop
then $e$ is not contained in any element of $\Sigma$, and in
this case
\[ \frac{1}{Q(e)}\, \sum_{x,\, e\in \sigma_x} \, \pi(x) =  0.\]
Therefore $\eta'(\Sigma) = \operatorname{max}_{x\in\Omega} P(x,x)^{-1}$
and the second statement follows from the first.
\end{proof}

We will use the \emph{multicommodity flow}
method of Sinclair~\cite{sinclair} to bound the second
eigenvalue $\lambda_1$.
A \emph{flow} in $\mathcal{G}$ is a function 
$f:\mathcal{P}\to [0,\infty)$ which satisfies
\begin{equation}
\label{flow-condition}
 \sum_{p\in\mathcal{P}_{xy}} f(p) = \pi(x)\pi(y)
   \quad \mbox{ for all } x,y\in \Omega, \, x\neq y,
\end{equation}
where $\mathcal{P}_{xy}$ is the set of all simple directed paths from
$x$ to $y$ in $\mathcal{G}$ and $\mathcal{P}= \cup_{x\neq y}\mathcal{P}_{xy}$.
Extend $f$ to a function on oriented edges by setting
\[ f(e) = \sum_{p\ni e} f(p),\]
so that $f(e)$ is the total flow routed through $e$.  
Let $\ell(f)$
be the length of the longest path with $f(p)>0$, and let
\[ \rho(e) = f(e)/Q(e)\]
be the \emph{load} of the edge $e$.  The 
\emph{maximum load} of the flow is
\[ \rho(f) = \max_{e} \rho(e).\]
Sinclair~\cite[Corollary $6'$]{sinclair} proves the following.

\begin{lemma} \emph{(\cite[Corollary $6'$]{sinclair})}
For any reversible ergodic Markov chain $\mathcal{M}$ and any flow $f$,
the second eigenvalue $\lambda_1$ satisfies
\[ (1-\lambda_1)^{-1} \leq  \rho(f)\ell(f).\]
\label{second-eigval}
\end{lemma}

\section{The switch chain}\label{s:markov}

Let $d:\{ 4,5,\ldots\}\rightarrow \mathbb{N}$ be any function such that
$1\leq d(n)\leq n-1$ for all $n\geq 4$, and denote by
$\Omega_{n,d(n)}$ 
the set of all $d(n)$-regular simple
digraphs with vertex set $[n]=\{ 1,2,\ldots, n\}$.  
We usually hide the dependence of $d$ on $n$ and just write
$d$ rather than $d(n)$;  similarly we write $\Omega_{n,d}$ for
$\Omega_{n,d(n)}$. 

We will study the Markov chain $\mathcal{M}$ described in
Figure~\ref{f:chain}, which we call the \emph{switch chain}.
From a given state, an unordered pair of two 
distinct arcs are chosen uniformly at random.  Then the
two chosen arcs exchange heads, unless the chosen arcs are incident or
exchanging their heads would create a repeated arc.
Note that two arcs are non-incident if and only if the set of endvertices
of the two arcs contains exactly four vertices.
\begin{figure}[ht]
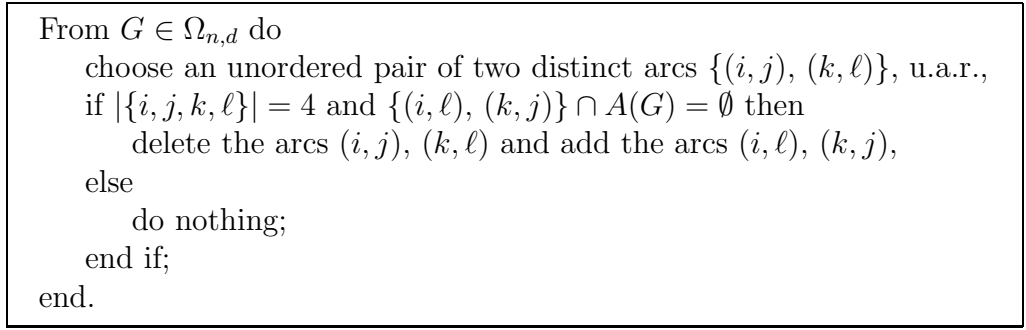

\begin{center}
\fbox{%
\begin{minipage}{17cm}
\begin{tabbing}
A\=\kill
\> From $G\in\Omega_{n,d}$ do\\
\> AA\= \kill
\> \> choose an unordered pair of 
          two distinct arcs $\{ (i,j)$, $(k,\ell)\}$, u.a.r., \hspace*{2mm}\\
\> \> if $|\{ i,j,k,\ell\}| = 4$ and 
            $\{ (i,\ell), \, (k,j)\} \cap A(G) = \emptyset$  then\\
\>\> AA\= \kill
\>\>\> delete the arcs $(i,j)$, $(k,\ell)$ 
                 and add the arcs $(i,\ell)$, $(k,j)$, \hspace*{2mm}\\
\> \> else\\
\> \>\> do nothing;\\
\> \> end if;\\
\>  end.
\end{tabbing}
\end{minipage}}
\end{center}
\caption{The Markov chain on $\Omega_{n,d}$}
\label{f:chain}
\end{figure}
We will write $[ijk\ell]$ as shorthand notation for the switch
that replaces the arcs $(i,j)$, $(k,\ell)$ with the arcs $(i,\ell)$,
$(k,j)$, as in Figure~\ref{f:chain}. 

The transition matrix $P$ of the Markov chain
satisfies $P(X,Y) = P(Y,X) = 1/\binom{dn}{2}$ if $X$ and $Y$ differ by
just a switch, with all other non-diagonal entries equal to zero.  
Therefore $P$ is symmetric, so the stationary distribution
of the Markov chain is uniform over $\Omega_{n,d}$.

It is not difficult to see that the switch chain is aperiodic, but
for completeness we give a brief proof.

\begin{lemma}
The switch chain on $\Omega_{n,d}$ is aperiodic
for $n\geq 4$ and $1\leq d\leq n-1$.
\label{aperiodic}
\end{lemma}

\begin{proof}
Fix $G\in \Omega_{n,d}$ and choose an arc $(\alpha,\beta)\in A(G)$.  
Since $d\geq 1$, there exists an arc $(\gamma,\alpha)\in A(G)$.
These two arcs are distinct but incident, and if they
are the arcs chosen in the transition procedure then the switch
will be rejected and the chain will remain at $G$. Hence
$P(G,G)\geq 1/\binom{dn}{2} > 0$.  So there is a self-loop at every
state of $\Omega_{n,d}$, which proves that the chain is aperiodic.
\end{proof}

Rao, Jana and Bandyopadhyay~\cite{RJB} showed that the switch chain is
not always irreducible
on the set of all digraphs with a given degree sequence.
Characterisations of degree sequences for which the chain is
irreducible were given in~\cite{BM,lamar}.
We will now prove that when $n\geq 4$ and $d\geq 1$
the set $\Omega_{n,d}$ is connected under switches;
that is, that the switch chain is irreducible on $\Omega_{n,d}$.
(This was already known when $d=1$,
see Diaconis, Graham and Holmes~\cite[Remark 2]{DGH}.)

We will use results from LaMar~\cite{lamar}.
For a set $U$ of vertices in a digraph $G$, 
define the sets $\C^{(i,j)} = \C^{(i,j)}(U,G)$
for ${(i,j)}\in \mathbb{Z}_2^2$, as follows:
\begin{align*}
\C^{(0,0)} &= \{ x\in V(G)- U : (x, u)\not\in A(G),\,
                 (u,x)\not\in A(G) \text{ for all } u\in U\},\\
\C^{(0,1)} &= \{ x\in V(G)- U : (x, u)\not\in A(G),\,
                 (u,x)\in A(G) \text{ for all } u\in U\},\\
\C^{(1,0)} &= \{ x\in V(G)- U : (x, u)\in A(G),\,
                 (u,x)\not\in A(G) \text{ for all } u\in U\},\\
\C^{(1,1)} &= \{ x\in V(G)- U : (x, u)\in A(G),\,
                 (u,x)\in A(G) \text{ for all } u\in U\}.
\end{align*}
(In~\cite{lamar} these sets are called $\mathcal{C}^0$, $\mathcal{C}^+$,
$\mathcal{C}^-$, $\mathcal{C}^{\pm}$, respectively.)

\begin{lemma}
\label{connected}
The space $\Omega_{n,d}$ is connected under switches when $n\geq 4$
and $1\leq d\leq n-1$.
\end{lemma}

\begin{proof}
For a contradiction, assume that $\Omega_{n,d}$ is not connected
under switches.  Then by LaMar~\cite[Theorems 3.3 and 3.4]{lamar},  
for any digraph $G\in\Omega_{n,d}$ there is a 
set of vertices $\{v_0,v_1,v_2\}$ such that the induced digraph
$G[\{v_0,v_1,v_2\}]$ is a directed 3-cycle and, writing
$\C^{(i,j)} = \C^{(i,j)}(\{ v_0,v_1,v_2\},G)$ for all 
${(i,j)}\in\mathbb{Z}_2^2$,
\begin{itemize}
\item[(i)] all vertices in $V(G)$ other than $\{ v_0, v_1, v_2\}$ belong to
${\bigcup_{{(i,j)}\in\mathbb{Z}_2^2} \C^{(i,j)}}$,
\item[(ii)] no arcs from  $\C^{(0,0)}\cup \C^{(0,1)}$  to  
$\C^{(0,0)}\cup\C^{(1,0)}$  are present, 
\item[(iii)] all (non-loop) arcs from  $\C^{(1,0)}\cup \C^{(1,1)}$ 
      to $\C^{(0,1)}\cup \C^{(1,1)}$  are present.
\end{itemize}
Let
$n^{(i,j)} = |\C^{(i,j)}|$ for ${(i,j)}\in \mathbb{Z}_2^2$.  
Considering the
in-degree and out-degree of $v_0$ gives, using (i),
\[ d = n^{(1,0)} + n^{(1,1)} + 1 = n^{(0,1)} + n^{(1,1)} + 1\]
(and in particular, $n^{(0,1)} = n^{(1,0)}$).  However, by (ii),
the in-degree of any element of $\C^{(1,0)}$  is at most
$n^{(1,0)} + n^{(1,1)} - 1 = d-2$,
the out-degree of any element of $\C^{(0,1)}$ is at most
$n^{(0,1)} + n^{(1,1)} - 1 = d-2$
and the out-degree of any element of $\C^{(0,0)}$ is at most 
$n^{(0,1)} + n^{(1,1)}  = d-1$.
This contradicts the assumption that $G\in\Omega_{n,d}$ unless
\[ \C^{(0,0)} \cup\C^{(0,1)}\cup\C^{(1,0)} =\emptyset.\]
But then $\C^{(1,1)} = V - \{ v_0,v_1,v_2\}$, and this set is
nonempty as $n\geq 4$.  By (iii), the in-degree of any element of 
$\C^{(1,1)}$ is $n-1$. Since the in-degree of $v_0$ is $n-2$,
we obtain a contradiction. 
\end{proof}

Suppose that $G\in\Omega_{n,d}$ contains a directed 3-cycle on the
vertices $v_0, v_1, v_2$.  Consider the sets 
$\C^{(i,j)}=\C^{(i,j)}(\{ v_0, v_1, v_2\},G)$
where ${(i,j)}\in\mathbb{Z}_2^2$.
If $x\in V(G) - \{ v_0,v_1,v_2\}$
does not belong to
$\bigcup_{{(i,j)}\in\mathbb{Z}_2^2}\C^{(i,j)}$
then we say that $x$ is a \emph{useful neighbour}
 for the given 3-cycle.
(Note that $x$ must be an in-neighbour or an out-neighbour of at least one
vertex on the 3-cycle, since $x\not\in\C^{(0,0)}$.)  Similarly,
$(x,y)$ is called a \emph{useful arc} for the given 3-cycle 
if $x\neq y$ and one of the following conditions holds:
\begin{itemize}
\item[(U1)] $(x,y)\in A(G)$, with
$x\in\C^{(0,0)}\cup\C^{(0,1)}$ and
$y\in\C^{(0,0)}\cup\C^{(1,0)}$;
\item[(U2)]
$(x,y)\not\in A(G)$, with
$x\in\C^{(1,0)}\cup\C^{(1,1)}$ and
$y\in\C^{(0,1)}\cup\C^{(1,1)}$.  
\end{itemize}
The following result will be used later.

\begin{lemma}
Suppose that $G\in\Omega_{n,d}$ contains a set of three vertices
$\{ v_0,v_1,v_2\}$ such that the induced digraph
$G[\{v_0,v_1,v_2\}]$ is a  directed 3-cycle. 
Then there exists a useful neighbour or a useful arc for 
this 3-cycle.
\label{useful}
\end{lemma}

\begin{proof}
Suppose that there is no useful neighbour for the 3-cycle. 
Then condition (i) from the proof of Lemma~\ref{connected} holds.
For a contradiction, assume that there is no useful arc $(x,y)$.
Then all (non-loop) arcs from $\C^{(0,0)}\cup\C^{(0,1)}$ to
$\C^{(0,0)}\cup\C^{(1,0)}$ are absent in $G$, and all
(non-loop) arcs from $\C^{(1,0)}\cup\C^{(1,1)}$ to $\C^{(0,1)}\cup \C^{(1,1)}$
are present in $G$. That is,
conditions (ii) and (iii) from the proof of Lemma~\ref{connected} also hold.  
Arguing as in the proof of Lemma~\ref{connected} leads to a contradiction.
\end{proof}

Now we prove a bound on the smallest eigenvalue of the switch chain.

\begin{lemma}
Suppose that $n\geq 4$ and $1\leq d\leq n-1$, and let $N = |\Omega_{n,d}|$.
The smallest eigenvalue $\lambda_{N-1}$ of the switch chain satisfies
\[ (1 + \lambda_{N-1})^{-1} \leq \dfrac{1}{4}\, d^2 n^2.\]
\label{our-smallest-eigval}
\end{lemma}

\begin{proof}
By Lemma~\ref{aperiodic}, there is a self-loop $\sigma_x$
in $\Gamma$ at every $x\in\Omega_{n,d}$.   
Let $\Sigma=\{\sigma_x\}$ be the set of these 1-cycles.  
Since $\ell(\Sigma) = 1$, the result follows
from the second statement of Corollary~\ref{congestion-corollary}
since
\[ \operatorname{max}_{x\in\Omega_{n,d}} P(x,x)^{-1} \leq \binom{dn}{2}.\]
\end{proof}

We also need a bound on the second-largest eigenvalue of the
switch chain.   

\begin{proposition}
Suppose that $n\geq 4$ and $1\leq d\leq n-1$, and let 
$\lambda_1$ be the second-largest eigenvalue of the
switch chain on $\Omega_{n,d}$.  Then
\[
(1 - \lambda_1)^{-1} \leq 50 d^{25} n^9.
\]
\label{our-second-largest-eigval}
\end{proposition}

The proof of Proposition~\ref{our-second-largest-eigval}
is lengthy and quite technical.  We give an outline of the
proof below, and full details in Sections~\ref{s:flow}--\ref{s:analysis}.
But first, we show how Theorem~\ref{main} can be proved
from Proposition~\ref{our-second-largest-eigval}.

\begin{proof}[Proof of Theorem~\ref{main}]\
If the smallest eigenvalue $\lambda_{N-1}$ is nonnegative
then $\lambda_* = \lambda_1$, and by 
Proposition~\ref{our-second-largest-eigval} we have  
\begin{equation}
 (1-\lambda_*)^{-1} \leq 50 d^{25} n^9.
\label{our-next-largest}
\end{equation}
Suppose now that $\lambda_{N-1}$ is negative. 
Then $1-|\lambda_{N-1}| = 1+\lambda_{N-1}$
and it follows from Lemma~\ref{our-smallest-eigval} and
Proposition~\ref{our-second-largest-eigval} that
(\ref{our-next-largest}) also holds in this case.

Finally, we note that
\begin{equation}
\label{log-omega}
\log |\Omega_{n,d}| \leq dn \log(dn).
\end{equation}
(This is well-known but for completeness we sketch a proof.  Take a 
bipartite graph on $n+n$ vertices and
assign $d$ ``half-edges'' to each vertex on the side.
Arbitrarily  match each half-edge on the left to a half-edge
on the right.  There are at most $(dn)^{dn}$ ways to perform this
matching.  Finally, orient each edge from left to right and
identify the $j$'th vertex on each side, giving a digraph on $n$
vertices which may have loops or multiple arcs.  As each element of
$\Omega_{n,d}$ can be formed from at least one matching in this
way, we obtain an upper bound.)
Hence, since $\pi$ is uniform,
\[ \log 1/\pi^\ast = \log{|\Omega_{n,d}}| \leq dn\log (dn).\]
Substituting (\ref{our-next-largest}) and (\ref{log-omega}) into
Lemma~\ref{mix-eigvals} gives the stated bound on the mixing time,
completing the proof.
\end{proof}

Hence it remains to establish Proposition~\ref{our-second-largest-eigval}.
We use a multicommodity flow argument to prove this result.
Before embarking on the proof, we outline the major steps in the
argument.  (We note that our proof follows the same general outline
as most canonical path or multicommodity flow arguments, where encodings
are often used.  In particular, our proof builds upon
the argument from~\cite{CDG}.)
\begin{itemize}
\item  Given distinct digraphs $G, G'\in\Omega_{n,d}$, we define a
finite set $\Psi(G,G')$ of objects, called pairings.   
For each $\psi\in\Psi(G,G')$
we will define a canonical path $\gamma_\psi(G,G')$ from $G$ to $G'$,
indexed by $\psi$.  Then the flow $f$ is defined on
\[ \bigcup_{(G,G')} \{ \gamma_\psi(G,G')\mid \psi\in\Psi(G,G')\}\]
by
\[ f(\gamma_\psi(G,G')) = \frac{\pi(G)\,\pi(G')}{|\Psi(G,G')|}
                        = \left(|\Omega_{n,d}|^2\, |\Psi(G,G')|\right)^{-1},
\]
and is set to zero for all other paths.
Note that $f$ satisfies (\ref{flow-condition}). 
\item
To define $\gamma_\psi(G,G')$ we work with the symmetric difference
$H = G\triangle G'$ of $G$ and $G'$ (with arcs of $G-G'$ coloured
blue and arcs of $G'-G$ coloured red).  In Sections~\ref{s:circuits}
and~\ref{s:circuit}
we show how to decompose $H$ into a sequence of arc-disjoint subdigraphs called
1-circuits and 2-circuits, in a canonical way.  The canonical path
$\gamma_\psi(G,G')$
is formed by processing each of these 1-circuits and 2-circuits in
the given order.
\item  We can process 1-circuits, and certain 2-circuits, in a way which
is very similar to the method used in~\cite{CDG} for undirected graphs.
The 2-circuits which can be handled in this way are called \emph{normal}.
Section~\ref{s:1circuit} explains how to process a 1-circuit and
Section~\ref{s:normal} describes how to process a normal 2-circuit.
\item  The main difficulties in the proof arise from the need to handle
2-circuits which are not normal.  We further categorise these as
\emph{eccentric 2-circuits} or \emph{triangles}.   Sections~\ref{s:eccentric}
and~\ref{s:triangle} describe how to process these 2-circuits.
\end{itemize}
By this stage, the multicommodity flow is completely defined.
Next we must analyse the flow in order to bound the maximum load of the
flow, and hence the second-largest eigenvalue (using 
Lemma~\ref{second-eigval}).
\begin{itemize}
\item Let $(Z,Z')$ be a transition along one of the canonical paths
$\gamma_\psi(G,G')$, and suppose that this transition is performed
while processing the 1-circuit or 2-circuit $S$.  A set of 
\emph{interesting arcs} for $Z$ with respect to $(G,G',\psi)$
is defined.
These are arcs which have been disturbed during the
processing of $S$ and not yet returned to their original state, 
and they will play a key role in our analysis.  Lemma~\ref{zoo}
describes the structure of the digraph formed by the interesting arcs
(see also Figure~\ref{f:zoo}).
\item  Next we identify $G$, $G'$ and $Z$ with their adjacency matrices
and define a matrix $L$ by $L+Z = G+G'$.  Then $L$ is an $n\times n$
matrix with entries in $\{ -1,0,1,2\}$.  We say that $L$ is 
an \emph{encoding} for $Z$ with respect to $(G,G')$.  Lemma~\ref{notquiteunique}
shows that given $(Z,Z')$, $L$ and $\psi$ there are at most four
possibilities for $(G,G')$ such that $(Z,Z')$ is a transition on
$\gamma_\psi(G,G')$ and $L$ is an encoding for $Z$ with respect
to $(G,G')$.   Further information about the structure of $L$ is
given in Lemma~\ref{structure}.
\item  Now the notion of encoding is broadened to encompass any $n\times n$
matrix with entries in $\{ -1,0,1,2\}$ such that all row sums and column
sums equal $d$.  Given $Z\in \Omega_{n,d}$, we say that the
encoding $L$ is $Z$\emph{-valid} if every entry of $L+Z$ belongs
to $\{ 0,1,2\}$ and $L, Z, H$ satisfy the conclusions of
Lemma~\ref{structure}.  (Here $H$ is the digraph defined by all
entries of $L+Z$ which equal 1.)  Lemma~\ref{Zvalid} proves a
useful fact about $Z$-valid encodings.
\item Next we explain how to apply switches to encodings, and prove
in Lemma~\ref{fix} (using Lemma~\ref{Zvalid})
that any $Z$-valid encoding can be transformed into
an element of $\Omega_{n,d}$ using at most three switches.
Counting the number of ways these switches can be performed in
reverse leads to an upper bound 
of the form $\operatorname{poly}(n,d)\, |\Omega_{n,d}|$
on the number of $Z$-valid encodings,
as proved in Lemma~\ref{poly}.
\item  Combining all this allows us to prove an upper bound on the
total flow routed through an arbitrary transition of the Markov chain.
This bound, of the form $\operatorname{poly}(n,d)\, |\Omega_{n,d}|^{-1}$,
is proved in Lemma~\ref{load}.  With this in hand it is easy to establish
a polynomial bound on the maximum load $\rho(f)$ of the flow, and
hence to prove Proposition~\ref{our-second-largest-eigval}.
\end{itemize}

\section{Defining the flow}\label{s:flow}

We now define the multicommodity flow which will be used
to bound the second largest eigenvalue, and hence the mixing time,
of the switch chain for regular directed graphs.

For $G,G'\in\Omega_{n,d}$, let $H = G\triangle G'$ be
the symmetric difference of $G$ and $G'$,  together with
an arc-colouring which colours all arcs of $G - G'$ blue
and all arcs of $G' - G$ red.  This arc colouring means
that we can think of $H$ as the symmetric difference of
the \emph{ordered pair} $(G,G')$.

For $v\in V$ let $\theta_v$ be the blue in-degree of $v$,
which equals the red in-degree of $v$, and let $\phi_v$
be the blue out-degree of $v$, which equals the red out-degree
of $v$.  Choose a \emph{pairing} of the red and blue arcs 
around each vertex as follows: each blue arc with head $v$ 
is paired with a red arc with head $v$, and
each blue arc with tail $v$ is paired with a red arc with 
tail $v$, defining two bijections (one from the set of blue arcs
with head $v$ to the set of red arcs with head $v$, and one
from the set of blue arcs with tail $v$ to the set of red arcs
with tail $v$).  Denote the set of all such pairings by
$\Psi(G,G')$.  Then
\begin{equation}
\label{number-pairings}
 |\Psi(G,G')| = \prod_{v\in V} \theta_v!\, \phi_v! 
\end{equation}
is the total number of pairings.

Write $\mathcal{G}$ for the underlying graph of the Markov chain
$\mathcal{M}$, where $\mathcal{G} = (\Omega_{n,d}, \Gamma)$
and each edge $e\in \Gamma$ corresponds to a transition of
$\mathcal{M}$.
For each pairing in $\Psi(G,G')$ we construct a canonical
path from $G$ to $G'$ in $\mathcal{G}$.
Each of these paths will carry $1/|\Psi(G,G')|$ of the total
flow from $G$ to $G'$.

We now introduce some terminology.  A 
\emph{forward circuit}
in $H$ is a string $C = w_0w_1\cdots w_{2k-1}$ over the alphabet $V$
such that the arcs
\begin{equation}
  (w_0,w_1),\, (w_2,w_1),\, (w_2,w_3),\, (w_4,w_3),\,
      \ldots (w_{2k-2},w_{2k-1}),\, (w_0,w_{2k-1})
\label{arclist}
\end{equation}
are all distinct, all belong to $A_H$ and alternate in colour:
that is, 
the arcs in 
\[ \{ (w_{2i},w_{2i+1}) : i=0,1,\ldots, k-1\}\]
all have one colour and the arcs  in
\[ \{ (w_{2i+2},w_{2i+1}) : i=0,1,\ldots,  k - 2\}\cup\{ (w_0,w_{2k-1})\}\] 
all have the other colour. 

The \emph{converse} of $H$ is the digraph obtained from $H$
by reversing the direction of every arc (but keeping the colours
the same).  
A \emph{reverse circuit}
 in $H$ is a string $C = w_0w_1\cdots w_{2k-1}$
over the alphabet $V$
which forms a forward circuit in the converse of $H$.
That is, the arcs
\begin{equation}
  (w_1,w_0),\, (w_1,w_2),\, (w_3,w_2),\, (w_3,w_4),\,
      \ldots (w_{2k-1},w_{2k-2}),\, (w_{2k-1},w_{0})
\label{reversearclist}
\end{equation}
are all distinct, all belong to $A_H$ and alternate in colour,
so that the arcs in 
\[ \{ (w_{2i+1},w_{2i}) : i=0,1,\ldots, k-1\}\]
all have one colour and the arcs in 
\[
\{  (w_{2i+1}, w_{2i+2}) : i=0,1,\ldots, k-2\}\cup \{ (w_{2k-1},w_0)\}
\]
all have the other colour. 
By \emph{circuit} we mean either a forward circuit or a reverse
circuit.  For a forward or reverse circuit $C$, denote by $A(C)$
the set of arcs in (\ref{arclist}) or (\ref{reversearclist}),
respectively. 
It is important to note that the arcs of a circuit  
alternate both in colour and orientation at each step.
While a circuit may contain both the arcs $(x,y)$ and $(y,x)$,
any three consecutive vertices on the circuit are distinct.

We now define two operations on digraphs. 
Let $\zeta$ denote the operation which takes a digraph to its converse
(that is, it reverses every arc in the digraph),
and let $\chi$ be the operation which takes a digraph to its complement.
Writing $[n]^{(2)}$ for the set of all ordered pairs of distinct
elements of $[n]$, the complement $\chi G$ of a digraph $G$ has
arc set $[n]^{(2)} - A(G)$.
Note that the operations $\zeta$ and $\chi$
commute and are both involutions.

We can also apply $\zeta$ and $\chi$ to the (arc-coloured) symmetric
difference $H=G\triangle G'$.  Here $\zeta H$ is the result of reversing
every arc in $H$, without changing the colour of any arc. 
Similarly, $\chi H$ is the result of exchanging the colour of every arc in $H$
(so that blue becomes red and vice-versa), without changing the
orientation of any arc.    To see this, note that the set of blue
arcs in $H=G\triangle G'$ equals
\[ A(G) - A(G')  = A(\chi G') - A(\chi G),\]
but this equals the set of red arcs in $(\chi G)\triangle (\chi G')$
(and similarly, the set of red arcs in $G\triangle G'$ equals the
set of blue arcs in $(\chi G)\triangle (\chi G')$).

Finally, we generalise these definitions so that they also
apply to (arc-coloured) sub-digraphs $U$ of $H$.  That is,
$\zeta U$ is the result of reversing every arc in $U$, without
changing the colour of any arc, while $\chi U$ is the result of
exchanging the colour of every arc in $U$ without changing the
orientation of any arc.

\subsection{Decomposition into circuits}\label{s:circuits}

Fix a pairing $\psi\in\Psi(G,G')$.  We decompose $H$ into
a sequence of circuits depending on $\psi$, as
follows.  Let $(w_0,w_1)$ be the lexicographically least
arc in $H$. Choose the arc $(w_2,w_1)$ which is paired with
$(w_0,w_1)$ at $w_1$.  (Note that if $(w_0,w_1)$ is blue then
$(w_2,w_1)$ is red, and vice-versa.  Furthermore, $w_2\neq w_0$
since $H$ is a symmetric difference.)  
Next choose the arc $(w_2,w_3)$
which is paired with $(w_2,w_1)$ at $w_2$.  (This arc will have the
same colour as $(w_0,w_1)$.)  Continue in this fashion.  Specifically, for
$i\geq 1$, if $w_{2i}\neq w_0$ then
let $(w_{2i}, w_{2i+1})$ be the arc which is paired with $(w_{2i},w_{2i-1})$
at $w_{2i}$ and let $(w_{2i+2},w_{2i+1})$ be the arc which is paired
with $(w_{2i},w_{2i+1})$ at $w_{2i+1}$.  The vertices $w_i$ are
not necessarily distinct, but the arcs are distinct.
The process terminates when $(w_0,w_{2k-1})$ is paired with $(w_0,w_1)$
at $w_0$, giving a forward circuit
$C_1 = w_0w_1\cdots w_{2k-1}$.

If $A(C_1) = A_H$ then $\mathcal{C} = \{ C_1\}$ and we are done.
Otherwise, take the lexicographically least arc not in $C_1$ and
generate a new circuit $C_2$ by the above procedure.  Continue generating
circuits until
\[ A_H = A(C_1)\cup A(C_2)\cup\cdots \cup A(C_s).\]
Then $\mathcal{C} = \{ C_1, C_2,\ldots, C_s\}$ and the circuits
$C_1, C_2,\ldots, C_s$ are arc-disjoint.
Note that, once the pairing has been chosen, $\mathcal{C}$ is formed without
regard to the colouring of $H$.  This property will be needed later.

Using $\mathcal{C}$, we form a path
\[ G = Z_0, Z_1,\ldots, Z_M = G'\]
from $G$ to $G'$ in the underlying graph of the Markov chain
(that is, to get from $Z_a$ to $Z_{a+1}$ we perform a switch).
The path is defined by processing each circuit $C_i$ in turn.
Processing a circuit changes its arcs 
from agreeing with $G$ to agreeing with $G'$, with no other arcs
being permanently altered  (though some may be temporarily changed while 
processing the circuit $C_i$).  The canonical path is defined
inductively.  If
\[ G = Z_0, Z_1,\ldots, Z_r\]
is the canonical path obtained by processing the first circuit $C_1$, and
\[ Z_r, Z_{r+1},\ldots, Z_N = G'\]
is the canonical path from $Z_r$ to $G'$ obtained by processing
the circuits $(C_2,\ldots, C_s)$, in order, then the canonical path
from $G$ to $G'$ corresponding to $\psi$ is given by the concatenation
of these two paths.

Thus it suffices to describe the canonical path corresponding to a 
particular circuit $C = w_0 w_1 \cdots w_{2k-1}$.
First we may need to decompose the circuit $C$ further.
A \emph{1-circuit} $S = v_0v_1v_2\cdots v_t$ is a string
on the letters $\{ w_0, w_1,\ldots, w_{2k-1}\}$
such that $v_0=w_0$ and $w_0$ appears only once in $S$.
Usually a 1-circuit will be a contiguous substring of $C$
(allowing reversal of direction and/or cyclic wrapping around $C$), but 
it may also contain one arc which is not an arc of $C$.

We now show how to decompose a circuit into a sequence of 1-circuits
(and possibly some single switches)
which will then be processed in order (as described in Section~\ref{s:1circuit})
to form the canonical path corresponding to $C$.

\subsection{Decomposition of a circuit }\label{s:circuit}

Given a circuit $C = w_0w_1\cdots w_{2k-1}$, let $v = w_0$
and let $C^{(0)}=C$ (which is the currently unprocessed segment of $C$).  
Suppose that the current digraph on the
canonical path from $G$ to $G'$ is $Z_J$.  If $w_i\neq v$ for $i=1,\ldots, 2k-1$
then $C^{(0)}$ is
a 1-circuit which we process (using the procedure described in Section~\ref{s:1circuit}), 
extending the canonical path as
\begin{equation}
\label{process}
G = Z_0,\ldots, Z_J, Z_{J+1},\ldots, Z_{J+t}.
\end{equation}
This completes the processing of $C$.  Otherwise,
$v$ appears $\theta$ times on $C^{(0)}$, 
where $2\leq \theta\leq \theta_v+\phi_v$.
Relabel the vertices on $C^{(0)}$ as
\[ C^{(0)} = v x_1\cdots y_1 v x_2 \cdots y_2 v
         \cdots v x_\theta \cdots y_\theta.\]
By construction, $C^{(0)}$ is a forward circuit, so $(v,x_1)\in A(C^{(0)})$.

Firstly, 
suppose that $S = v x_1 \cdots y_1$ is a 1-circuit.  That is, arcs
$(v,x_1)$ and $(v,y_1)$ are present on $S$, with opposite colours.
Process the 1-circuit $S$, extending the canonical path as in 
(\ref{process}), leaving the forward circuit
\[ C^{(1)} = v x_2\cdots y_2 v \cdots v x_\theta \cdots y_\theta\]
as the unprocessed section of $C$.  Then process $C^{(1)}$ inductively.

Next, suppose that we are not in the above situation (so that
the arcs $(v,x_1)$ and $(y_1,v)$ are present on $S$, with the same colour),
but that $S = vx_\theta \cdots y_\theta$ is a 1-circuit.
That is, the arcs $(v,x_\theta)$ and $(v,y_\theta)$ are present on $S$,
with opposite colours.
Process the 1-circuit $S$, extending the canonical path as in (\ref{process}), 
and leaving the forward circuit
\[ C^{(1)} = vx_1\cdots y_1 v \cdots v x_{\theta-1} \cdots y_{\theta-1}\]
to be processed inductively.  

Finally, suppose that neither of the two situations above apply.
Then the arcs $(v,x_1)$ and $(y_1,v)$ have one colour 
while $(x_\theta, v)$ and $(v, y_\theta)$ have the other colour.
We will process 
$S' = v x_1\cdots y_1 v x_\theta \cdots y_\theta $ (which we 
call a \emph{2-circuit}),
extending the canonical path as in (\ref{process})
and leaving
\[ C^{(1)} = v x_2 \cdots y_2 v \cdots v x_{\theta-1} \cdots y_{\theta-1}
\]
to be processed inductively.  Here $C^{(1)}$ is a reverse circuit
and we process it using the same procedure as described above, but with
all arcs reversed.

All 1-circuits and 2-circuits created by the above procedure are
called \emph{raw}.
The order in which we detect and process raw 1-circuits and raw
2-circuits implies
that both the processed and unprocessed sections of $C$ are contiguous
whenever the processing of a raw 1-circuit or raw 2-circuit is complete.
(That is, these sections form contiguous substrings of $C$, where a substring
is allowed to wrap around in a cyclic fashion.)

Suppose that $S$ is a raw 1-circuit or raw 2-circuit with
substring $abc$.  Fix $i\in \{ 0,1\}$ such that the corresponding
arcs are 
$\zeta^i(a,b)$ and $\zeta^i(c,b)$.  These arcs are called
\emph{successive arcs} along $S$.
Every raw 1-circuit or raw 2-circuit $S$ has the following property:  
successive arcs along $S$ are paired under $\psi$ at their
common endvertex $b$, except possibly when 
$b=v$ and the arcs are the first and last arcs of $S$.
We call this the \emph{well-paired} property,
which will be used in
Lemma~\ref{simplify} below.

Raw 1-circuits are processed using the method described in 
Section~\ref{s:1circuit}.  In most cases, raw 2-circuits must be further
decomposed (into a sequence of 1-circuits and/or switches)
before they can be processed, as described in 
Section~\ref{2circuit}.  It is here that extra
difficulties arise here when working with directed graphs.

Recall the notation for switches introduced after
Figure~\ref{f:chain}. 
Let $Q=\alpha\beta\gamma\delta$ be a circuit in $H$ which is
also a 4-cycle.  Set 
\[ i=\begin{cases} 0 & \text{ if $Q$ is a forward circuit,}\\
             1 & \text{ otherwise}.
\end{cases}
\]  
We now define notation for the switch which processes this 4-cycle,
starting from the current digraph $Z_J$ and producing the next digraph 
$Z_{J+1}$ on the canonical path.
Let $h=0$ if $\zeta^i (\alpha,\beta)\in Z_J$ and $h=1$ otherwise.
Then define 
\[ \zeta^i\chi^h [\alpha\beta\gamma\delta] = 
   \begin{cases} [\alpha\beta\gamma\delta] & \text{ if $i=0$, $h=0$,}\\ 
                 [\alpha\delta\gamma\beta] & \text{ if $i=0$, $h=1$,}\\ 
                 [\beta\alpha\delta\gamma] & \text{ if $i=1$, $h=0$,}\\ 
                 [\beta\gamma\delta\alpha] & \text{ if $i=1$, $h=1$.}
 \end{cases}
\]
If $h=0$ then the switch $\zeta^i\chi^h [\alpha\beta\gamma\delta]$ 
deletes the arcs
$\zeta^i(\alpha,\beta)$, $\zeta^i(\gamma,\delta)$ and replaces them
with $\zeta^i(\alpha,\delta)$, $\zeta^i(\gamma,\beta)$, while the
opposite occurs if $h=1$.

Finally, we define the \emph{status} of an arc $(x,y)$
in a digraph $Z$ to equal 0 if $(x,y)\not\in A(Z)$ and to equal 1 if
$(x,y)\in A(Z)$.  We say that two arcs have 
\emph{matching status} if their status is
equal in $Z$, and say that they have 
\emph{opposite status} otherwise.

\subsection{Processing a 1-circuit}\label{s:1circuit}

Let $S$ be a 1-circuit.  (If $S$ is not raw then $S$ has resulted
from the decomposition of a raw 2-circuit: see Sections~\ref{s:normal} 
and~\ref{s:eccentric}.)
The method for processing a 1-circuit is very similar to 
that used in~\cite{CDG}, and some of the discussion and figures given there
may be helpful.
(See also the worked example in Section~\ref{a:example}.)

Label the 1-circuit as $S = x_0 x_1\ldots x_{2k-1}$
where $k\geq 2$, such that $x_0$ is the minimum vertex on $S$
and $x_1 = \min\{ x_1,\, x_{2k-1}\}$.  Set 
\[ i=\begin{cases} 0 & \text{if $S$ is a forward circuit,}\\
                    1 & \text{if $S$ is a reverse circuit.}
\end{cases}
\]
Also set 
\[ h= \begin{cases} 0 & \text{if $\zeta^i (x_0,x_1)\in A(Z_J)$,}\\
      1 & \text{ otherwise.}
      \end{cases}
      \]
Then
$\zeta^i (x_{2t},x_{2t+1})\in A(\chi^h Z_J)$ and
$\zeta^i (x_{2t+2}, x_{2t+1})\not\in A(\chi^h Z_J)$ for $t=0,1,\ldots, k-1$
(identifying $x_{2t}$ with $x_0$).   Note that
any three consecutive vertices on $S$ are distinct.

Define the set
\begin{align*}
 \mathcal{B} &= \{ t =  1,2,\ldots, k-1 \,\, :\,\,  \zeta^i (x_0,x_{2t+1})\not\in
     A(\chi^h Z_J)  \\
         & \qquad \qquad
	   \mbox{ and } x_{2\ell+1} \neq x_{2t+1} \mbox{ for }
             \ell \mbox{ with } \ell = t+1,\ldots, k-1. \}
\end{align*}
(This definition ensures that exactly one value $t$ is stored for each
distinct vertex $x_{2t+1}$ with $\zeta^i (x_0,x_{2t+1})\not\in A(\chi^h Z_J)$,
ensuring that vertices which are repeated along $S$ are treated correctly.)
Note that $k-1\in\mathcal{B}$ always.
The arcs $\zeta^i (x_0,x_{2t+1})$ are called \emph{odd chords}.
The number of phases in the processing of $S$ will be
$p = |\mathcal{B}|$.  For the first phase, choose the minimum $t\in\mathcal{B}$.
There will be $t$ steps in the first phase, which proceeds as follows:  
\begin{center}
\begin{tabbing}
for $j:= t, \, t-1, \ldots, 1$  do\\
XXX \= \kill
\> form $Z_{J+t-j+1}$ from $Z_{J+t-j}$ by performing the 
switch $\zeta^i\chi^h[x_0x_{2j-1}x_{2j}x_{2j+1}]$;\\
\end{tabbing}
\end{center}
If $t=k-1$ then there is only one phase and the processing of $S$ is complete.
Otherwise, $\zeta^i (x_0, x_{2t+1}) \in A(\chi^h Z_{J+t})$ but all odd
chords $\zeta^i (x_0,x_{2\ell+1})$ with $x_{2\ell+1}\neq x_{2t+1}$
have been reinstated to match their status in $Z_J$ (that is,
they belong to $Z_{J+t}$ if and only if they belong to $Z_J$).

For subsequent phases, if $t$ was the starting point of the previous phase
then choose $q > t$ minimum such that $q\in\mathcal{B}$.  
The odd chord $\zeta^i (x_0,x_{2t+1})$ has been switched in the previous
phase but will be restored to its original state by the end of this phase.
There will be
$q-t$ steps in this phase, performing the sequence of switches
\[ \zeta^i\chi^h[x_0 x_{2q-1} x_{2q} x_{2q+1}],\,\, 
    \zeta^i\chi^h[x_0 x_{2q-3} x_{2q-2} x_{2q-1}],\,\, \ldots
  , \,\, \zeta^i\chi^h[x_0 x_{2t+1} x_{2t+2} x_{2t+3}].\]
Note that every switch involves $x_0$, the start-vertex of $S$.

At any point during the processing of the 1-circuit, at most three odd chords 
have been switched (that is, temporarily disturbed).  This is illustrated 
in the worked example in Section~\ref{a:example}.

\section{Decomposition of a raw 2-circuit }\label{s:2circuit}

We now show how to process a raw 2-circuit $S$, given the method
for processing a 1-circuit described in Section~\ref{s:1circuit}.
Suppose that we have reached the digraph $Z_J$ on the canonical path
from $G$ to $G'$ in $\mathcal{G}$.  Note that $Z_J$
agrees with $G$ on the 2-circuit $S$ before the processing of $S$ begins.
We relabel the vertices on $S$ as 
\begin{equation}
\label{Slabels}
S = v  x_{0,0}   \cdots x_{1,0}  v  x_{1,1} \cdots  x_{0,1}, 
\end{equation}
where $(v,x_{0,0})$ is the lexicographically least arc in $A(S)$.
Treat the indices $(i,j)$ on the vertex labels as elements of
$\mathbb{Z}_2^2$, with addition performed modulo 2.
In the undirected case~\cite{CDG}, the vertices
$x_{0,0}, \, x_{0,1}, \, x_{1,0}, \, x_{1,1}$
are all distinct.  However, this is no longer the case in the directed
setting, which complicates the definition of the canonical paths.
Also note that there may be as few as two vertices between two successive
occurrences of $v$ on the 2-circuit $S$.  This is explained in more
detail below Figure~\ref{2circuit}, once we have introduced some
useful notation.

Recall that $\chi$ is the complementation operation
for digraphs.  Set 
\[ h = \begin{cases} 
        0 & \text{ if the arc $(v,\, x_{0,0})$ is present in $Z_J$,}\\
        1 & \text{ if the arc $(v,\, x_{0,0})$ is absent in $Z_J$.}
        \end{cases}
\]
Then 
\[ (v,\, x_{0,0})\in A(\chi^h Z_J), \,\,\,
   (x_{1,0},\, v)\in A(\chi^h Z_J),\,\,\,
   (x_{1,1},\, v)\not\in A(\chi^h Z_J),\,\,\,
    (v,\, x_{0,1})\not\in A(\chi^h Z_J).
   \]
Figure~\ref{2circuit-lite} depicts $\chi^h S$,
where the curved lines (from $x_{0,0}$ to $x_{1,0}$ and from 
$x_{0,1}$ to $x_{1,1}$)
represent any odd number of alternating arcs.  Solid arcs represent 
arcs which are present in $\chi^h Z_J$ and dashed arcs represent arcs 
which are absent in $\chi^h Z_J$.  That is, if $h=0$ then
solid arcs belong to $Z_J$ and dashed arcs belong to $G'$,
while if $h=1$ then solid arcs belong to $G'$ and dashed arcs belong to
$Z_J$.
\begin{center}
\begin{figure}[ht]
 \psfrag{v}{$v$}\psfrag{x1}{$x_{0,0}$}\psfrag{w1}{$y_{0,0}$}
 \psfrag{z1}{$y_{1,0}$}\psfrag{y1}{$x_{1,0}$}\psfrag{x2}{$x_{1,1}$}
 \psfrag{w2}{$y_{1,1}$}\psfrag{z2}{$y_{0,1}$}\psfrag{y2}{$x_{0,1}$}
\centerline{\includegraphics[scale=0.6]{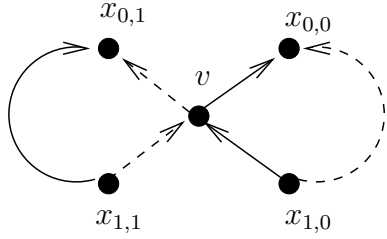}}
\caption{The 2-circuit $\chi^h S$}
\label{2circuit-lite}
\end{figure}
\end{center}
For $(i,j)\in\mathbb{Z}_2^2$, let $y_{i,j}$ be the unique vertex such that 
$v x_{i,j} y_{i,j}$ or $y_{i,j} x_{i,j} v$ is a contiguous substring of $S$
(allowing cyclic wrapping in the case of $y_{0,1}$). 
If $y_{0,j}=x_{1,j}$ for some $j$  then
$y_{1,j} = x_{0,j}$ and there is only one arc between $x_{0,j}$
and $x_{1,j}$.  This means that the corresponding curved line in
Figure~\ref{2circuit-lite} can be replaced by a single arc.  
There are four possibilities for $\chi^h S$ in which 
$y_{0,j}= x_{1,j}$ for $j=1,2$.
These are shown in Figure~\ref{bowtie-family}. 
The leftmost 2-circuit involves 5 distinct vertices
and the middle two 2-circuits each involve 4 distinct
vertices, with one coincidence of the form $x_{0,j}=x_{1,j+1}$,
where $j\in \mathbb{Z}_2$.
The rightmost 2-circuit 
involves 3 distinct vertices:  we will call it  a 
\emph{triangle}.
\begin{center}
\begin{figure}[ht]
 \psfrag{v}{$v$}\psfrag{x1}{$x_{0,0}$}\psfrag{w1}{$y_{0,0}$}
 \psfrag{y1}{$x_{1,0}$}\psfrag{x2}[c]{$x_{1,1}$} \psfrag{y2}[c]{$x_{0,1}$}
\psfrag{w1}[c]{$x_{0,0} = x_{1,1}$}\psfrag{z1}[c]{$x_{0,1} = x_{1,0}$}
\centerline{\includegraphics[scale=0.6]{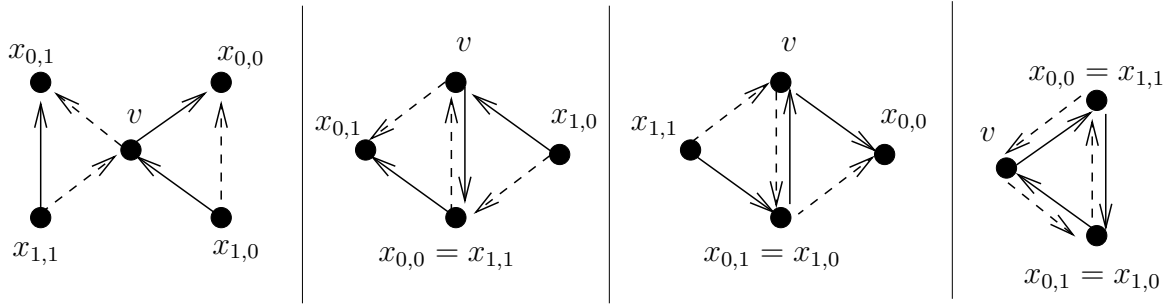}}
\caption{The four 2-circuits $\chi^h S$ with at most 5 distinct vertices.
The rightmost 2-circuit is a triangle.}
\label{bowtie-family}
\end{figure}
\end{center}
In the undirected analysis~\cite{CDG}, a critical observation was
that vertex $y_{0,0}$ must be distinct from $x_{0,1}$ (without loss of
generality).   This fact underpinned the definition
of the canonical paths in~\cite{CDG}.  For directed graphs this
property does not necessarily hold, as can be seen from the 
last two 2-circuits in Figure~\ref{bowtie-family}.

We will say that $S$ is \emph{normal} if $y_{i,j}\neq x_{i,j+1}$
for some $(i,j)\in\mathbb{Z}_2^2$.  
In Section~\ref{s:normal} we describe how to process a normal 2-circuit.
The procedure is analogous to that used in~\cite{CDG},  which is the motivation
for
the definition of normal 2-circuits.   Note that the triangle 
(shown at the rightmost of Figure~\ref{bowtie-family}) is not normal but
the remaining 2-circuits in Figure~\ref{bowtie-family} are normal. 

For $(i,j)\in\mathbb{Z}_2^2$, let $z_{i,j}$ be the unique vertex
such that  $v x_{i,j} y_{i,j} z_{i,j}$ or 
$z_{i,j} y_{i,j} x_{i,j}  v$
is a contiguous substring of $S$ (allowing cyclic wrapping in the case of $z_{0,1}$).  
We will need the following lemma.

\begin{lemma}
Suppose that $S$ is a raw 2-circuit which is not normal and 
such that
$v=z_{i,j}$ for some $(i,j)\in\mathbb{Z}_2^2$. 
Then  $S$ is a triangle.
\label{simplify}
\end{lemma}

\begin{proof}
Without loss of generality (by reversing arcs and/or taking the complement if necessary)
we may suppose that $v=z_{0,0}$.  
(This means we cannot assume that $(v,x_{0,0})$ is the lexicographically
least arc in $S$, but we do not need to use that property in this proof.)
Colour arcs around the 2-circuit orange, purple in an
alternating fashion, starting with the orange arc $(v,x_{0,0})$. 
By assumption, $S$ has initial substring
$v\, x_{0,0}\, y_{0,0}\, v$.  By the well-paired property of 2-circuits,
we know that the orange arc $(v,x_{0,0})$ is paired with the purple arc
$(y_{0,0},x_{0,0})$ at $x_{0,0}$ under $\psi$, and the purple arc
$(y_{0,0},x_{0,0})$ is paired with the orange arc $(y_{0,0},v)$ at $y_{0,0}$ 
under $\psi$.  
Now $v$ is incident with exactly four arcs of $S$, one of each colour and orientation
(see Figure~\ref{2circuit-lite}.)
Hence the presence of the orange arc $(y_{0,0},v)$ on $S$
shows that $y_{0,0} = x_{1,0}$.  But then we obtain
\[ y_{0,0} = x_{1,0} = y_{1,1} = x_{0,1},\]
as $S$ is not normal.

Now $y_{0,0}=x_{1,0}$ and the purple arc $(y_{0,0},x_{0,0})$ is paired with
the orange arc $(v,y_{0,0})$.  This implies that $x_{0,0}=y_{1,0}$, and since
$S$ is not normal it follows that
\[ x_{0,0} = y_{1,0} = x_{1,1} = y_{0,1}.\]
This gives all pairing information around the 2-circuit
except for pairings at $v$.   (For example, since $y_{0,0}=x_{0,1}$
and $x_{0,0}=y_{0,1}$, we know that the purple arc $(v,y_{0,0})$
is present on $S$ and is paired with the orange arc $(x_{0,0},y_{0,0})$
at $y_{0,0}$.  This arc is paired at $x_{0,0}$ with the purple arc
$(x_{0,0},v)$, since $x_{0,0} = x_{1,1}$ and $y_{0,0}=y_{1,1}$.)
But as $v$ is only incident
with four arcs of $S$ there can be no other vertices involved in $S$. 
So by the well-paired property, at least one of the pairs of arcs 
$(v,x_{0,0}), (v,y_{0,0})$
and $(x_{0,0},v),(y_{0,0})$ must be paired at $v$. 
It follows that $S$ is a triangle on $\{ v, x_{0,0}, y_{0,0}\}$.  
\end{proof}

Call $S$ \emph{eccentric}
if it is not normal and not a triangle.
If $S$ is eccentric then $v\neq z_{i,j}$ for all $(i,j)\in\mathbb{Z}_2^2$, 
by Lemma~\ref{simplify}.  Hence $\chi^h S$ is as shown in 
Figure~\ref{eccentric}.  (Remember that arcs must alternate in both
colour and orientation, giving a unique way to navigate around this figure,
or see Figure~\ref{eccentric-detail} below for an unravelled version.)
Again the curved lines represent an odd
number of alternating arcs (from $x_{0,0}$ to $x_{1,0}$ and
from $x_{0,1}$ to $x_{1,1}$).  Recall also that the vertices $x_{i,j}$ are
not necessarily distinct.
\begin{center}
\begin{figure}[ht]
 \psfrag{v}{$v$}\psfrag{x1}{$x_{0,0}=y_{0,1}$}\psfrag{w1}{$x_{0,1}=y_{0,0}$}
 \psfrag{y1}[r]{$x_{1,0}=y_{1,1}$}\psfrag{z1}[r]{$x_{1,1}=y_{1,0}$}
 \psfrag{p1}{$z_{0,0}$} \psfrag{p2}{$z_{1,1}$}\psfrag{q1}{$z_{1,0}$}
 \psfrag{q2}{$z_{0,1}$}
\centerline{\includegraphics[scale=0.6]{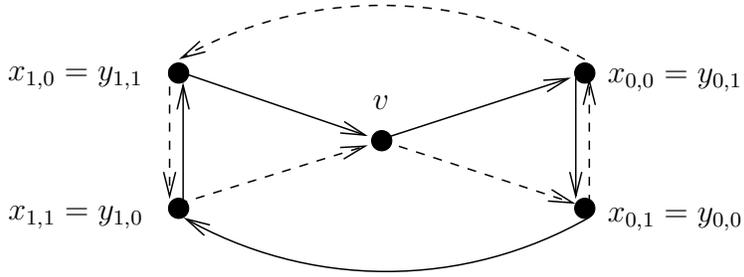}}
\caption{The 2-circuit $\chi^h S$ when $S$ is eccentric}
\label{eccentric}
\end{figure}
\end{center}

We describe how to process an eccentric
2-circuit in Section~\ref{s:eccentric} and in Section~\ref{s:triangle}
we explain how to process a triangle.  This will complete the
description of the canonical path from $G$ to $G'$ corresponding
to the pairing $\psi$.

\subsection{Decomposing a normal 2-circuit}\label{s:normal}

Let $S$ be a normal 2-circuit, with vertices labelled
as in (\ref{Slabels}), where $(v,x_{0,0})$ is the
lexicographically least arc in $A(S)$.
Recall the notation $z_{i,j}$ defined before Lemma~\ref{simplify}.
A normal 2-circuit was depicted in Figure~\ref{2circuit-lite} but now
we need a more detailed picture (Figure~\ref{2circuit}).
Recall however that there can be as few as three arcs in the left or right
half of this figure: for example, if there were only three arcs on the
right then $y_{i,0}=x_{i+1,0}$ and $z_{i,0}=v$ for $i\in\mathbb{Z}_2$.
Again the curved lines
in Figure~\ref{2circuit} represent an odd number of alternating arcs.
\begin{center}
\begin{figure}[ht]
 \psfrag{v}{$v$}\psfrag{x00}{$x_{0,0}$}\psfrag{y00}{$y_{0,0}$}
 \psfrag{y10}{$y_{1,0}$}\psfrag{x10}{$x_{1,0}$}\psfrag{x11}{$x_{1,1}$}
 \psfrag{y11}{$y_{1,1}$}\psfrag{y01}{$y_{0,1}$}\psfrag{x01}{$x_{0,1}$}
 \psfrag{a11}{$z_{1,1}$}\psfrag{a01}{$z_{0,1}$}\psfrag{a10}{$z_{1,0}$}
\psfrag{a00}{$z_{0,0}$}
\centerline{\includegraphics[scale=0.6]{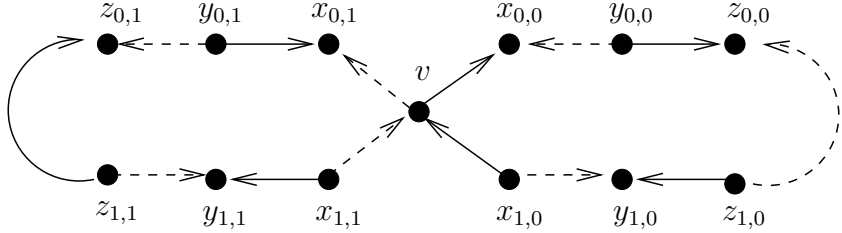}}
\caption{A normal 2-circuit $\chi^h S$, in more detail}
\label{2circuit}
\end{figure}
\end{center}
Let $(i,j)$ be the lexicographically
least index such that $x_{i,j}\neq y_{i,j+1}$.
(Here we use the ordering $0<1$ on $\mathbb{Z}_2$.)
Define the arc $a_{i,j} = (y_{i,j+1},\, x_{i,j})$.  The \emph{shortcut arc}
of $S$ is $\zeta^i a_{i,j}$ (that is, it equals $a_{i,j}$ itself if
$i=0$ and equals the reversal of $a_{i,j}$ if $i=1$).

Suppose that $Z_J$ is the current digraph on the canonical path from $G$ to $G'$
before we start decomposing $S$.
There are three cases, called (Na), (Nb), (Nc), where the `N' stands for `normal'.
\begin{enumerate}
\item[(Na)] the shortcut arc $\zeta^i a_{i,j}$ belongs to $A(S)$.
\item[(Nb)] the shortcut arc $\zeta^i a_{i,j}$ does not belong to $A(S)$, 
  and $\zeta^i a_{i,j}$ is not an arc of $\chi^{h + j} Z_J$.
\item[(Nc)] the shortcut arc $\zeta^i a_{i,j}$ does not belong to $A(S)$,
  and $\zeta^i a_{i,j}$ is an arc of $\chi^{h + j} Z_J$.
\end{enumerate}
We consider these cases in order.  (A more detailed description of the
analogous process in the undirected case, with figures, can be found
in~\cite{CDG} and may also be helpful.)
\begin{enumerate}
\item[(Na)]
In case (Na), the 2-circuit $S$ can be split into two 1-circuits, $S_1$
and $S_2$.   There are four
subcases to consider, depending on which ``half'' of the 2-circuit contains
the shortcut arc and whether the shortcut arc belongs to $Z_J$.
In all subcases, the arcs of $S_1$ and $S_2$ form a partition of the arcs of $S$.

Once the two 1-circuits $S_1$ and $S_2$ have been identified, they are processed 
in that order, 
extending the canonical path from $G$ to $G'$ as
\[ G = Z_0,\ldots, Z_J, Z_{J+1},\ldots, Z_{J+k}\]
after processing $S_1$, and
\[ G = Z_0,\ldots, Z_J, Z_{J+1}, \ldots, Z_{J+k}, Z_{J+k+1},\ldots, Z_{J+k+\ell}\]
after processing $S_2$.
\begin{enumerate}
\item[(Na1)] Suppose that $S$ can be rewritten (allowing cyclic wrapping if necessary)
as
\[ 
v\, x_{i,j+1}, y_{i,j+1}  z_{i,j+1} \cdots y_{i,j+1}  x_{i,j} \cdots 
     z_{i+1,j+1}  y_{i+1,j+1} x_{i+1,j+1}  v  x_{i+1,j} 
        \cdots x_{i,j}\]
and $\zeta^i a_{i,j} \not\in A(\chi^{h+j} Z_J)$.  Split $S$ into two 1-circuits
\begin{align*}
S_1 &= v x_{i,j+1}\,  y_{i,j+1}  z_{i,j+1} \cdots y_{i,j+1}  x_{i,j}, \\
S_2 &= v  x_{i+1,j+1}  y_{i+1,j+1}  z_{i+1,j+1}\cdots x_{i,j}  y_{i,j} 
  \cdots y_{i+1,j}  x_{i+1,j}.
\end{align*}
\item[(Na2)]  Suppose that $S$ can be rewritten (allowing cyclic wrapping if necessary)
as
\[ 
v\, x_{i,j+1}  y_{i,j+1}  z_{i,j+1} \cdots x_{i,j}\,  y_{i,j+1} \cdots 
       z_{i+1,j+1}  y_{i+1,j+1}  x_{i+1,j+1}  v\, 
        x_{i+1,j} \cdots x_{i,j}\]
and $\zeta^i a_{i,j} \in A(\chi^{h+j} Z_J)$. Split $S$ into two 1-circuits
\begin{align*}
S_1 &= v x_{i,j+1}  y_{i,j+1}  z_{i,j+1} \cdots  x_{i,j},\\
S_2 &= v  x_{i+1,j+1}  y_{i+1,j+1}  z_{i+1,j+1}\cdots 
          y_{i,j+1}\,  x_{i,j} y_{i,j} \cdots y_{i+1,j}   x_{i+1,j}.
		 \end{align*}
\item[(Na3)]
Suppose that $S$ can be rewritten (allowing cyclic wrapping if necessary)
as
\[  v x_{i,j+1}\cdots x_{i+1,j+1}  v  x_{i+1,j}  y_{i+1,j}
         z_{i+1,j} \cdots y_{i,j+1}  x_{i,j} 
	  \cdots z_{i,j+1}  y_{i,j+1}  x_{i,j+1} \]
and $\zeta^i a_{i,j} \not\in A(\chi^{h+j} Z_J)$.  Split $S$ into two 1-circuits
\begin{align*}
S_1 &= v  x_{i,j+1}  y_{i,j+1}  x_{i,j}\cdots  
           z_{i,j}  y_{i,j}  x_{i,j},\\
S_2 &= v  x_{i+1,j+1}  y_{i+1,j+1} \cdots y_{i,j+1} \cdots z_{i+1,j}
            y_{i+1,j}  x_{i+1,j}.
\end{align*}
\item[(Na4)]
Suppose that $S$ can be rewritten (allowing cyclic wrapping if necessary)
as
\[  v x_{i,j+1}\cdots x_{i+1,j+1}  v  x_{i+1,j}  y_{i+1,j}  
     z_{i+1,j} \cdots x_{i,j}  y_{i,j+1} \cdots z_{i,j}  y_{i,j}  x_{i,j}\]
and $\zeta^i a_{i,j} \in A(\chi^{h+j} Z_J)$.  Split $S$ into two 1-circuits
\begin{align*}
S_1 &= v  x_{i,j+1}  y_{i,j+1}\cdots z_{i,j}  y_{i,j}  x_{i,j},\\
S_2 &= v  x_{i+1,j+1}  y_{i+1,j+1} \cdots y_{i,j+1} x_{i,j} \cdots 
             z_{i+1,j}  y_{i+1,j}  x_{i+1,j}.
	     \end{align*}
\end{enumerate}
\item[(Nb)]
Now suppose that $S$ is a normal 2-circuit, the shortcut arc $\zeta^i a_{i,j}$ 
is not an arc of $S$
and $\zeta^i a_{i,j}$ is not an arc of $\chi^{h+j} Z_J$.   Then we can 
use the shortcut arc to give
an alternating 4-cycle $v  x_{i,j}  y_{i,j+1}  x_{i,j+1}$.
First process this alternating 4-cycle  using the switch 
$\zeta^i\chi^{h+j} [ v x_{i,j}, y_{i,j+1} x_{i,j+1}]$,
extending the canonical path by one step to give
\[ G = Z_0, \ldots, Z_J, Z_{J+1}. \]
(Call this step the \emph{shortcut switch}.)
Now $\zeta^i a_{i,j}$ is an arc of $\chi^{h+j} Z_{J+1}$ and we
can form a 1-circuit $S_1$ from $S$, specifically
\begin{equation}
\label{normalS1}
S_1 = v  x_{i+1,j+1}  y_{i+1,j+1} \cdots y_{i,j+1}\,  x_{i,j}  
         y_{i,j} \cdots y_{i+1,j}  x_{i+1,j}.
\end{equation}
Process this 1-circuit (as described in Section~\ref{s:1circuit})
to extend the canonical path further, giving
\[ G = Z_0,\ldots, Z_J, Z_{J+1}, Z_{J+2},\ldots, Z_{J+k}.\]
Note that $\zeta^i a_{i,j}$ is not an arc of $\chi^{h + j} Z_{J+k}$
after the 1-circuit $S_1$ has been processed, so it has been
restored to the same state as in $\chi^{h+j} Z_J$, before the processing of
the 2-circuit $S$ began.
\item[(Nc)]  Finally assume that $S$ is a normal 2-circuit, the shortcut
arc $\zeta^i a_{i,j}$ is not an arc of $S$ and $\zeta^i a_{i,j}$ is an arc of 
$\chi^{h + j} Z_J$.   Then the shortcut arc completes the 1-circuit $S_1$ 
defined in (\ref{normalS1}), 
which is processed (as described in Section~\ref{s:1circuit}).
This extends the canonical path to give
\[ G = Z_0, \ldots, Z_J, Z_{J+1}, \ldots, Z_{J+k}.\]
Last we process the alternating 4-cycle $v x_{i,j} y_{i,j+1} x_{i,j+1}$, 
using the shortcut switch $\zeta^i\chi^{h+j} [ v x_{i,j}, y_{i,j+1} x_{i,j+1}]$,
extending the canonical path by one step to give
\[ G = Z_0, \ldots, Z_J, Z_{J+1}, \ldots, Z_{J+k}, Z_{J+k+1}.\]
Note that $\zeta^i a_{i,j}$ is not an arc of $\chi^{h + j} Z_{J+k}$
but it is an arc of $\chi^{h+j} Z_{J+k+1}$, so it has been restored to 
the same state as in $\chi^{h+j} Z_J$.
\end{enumerate}

\subsection{Decomposing an eccentric 2-circuit}\label{s:eccentric}

Now we may assume that $S$ is an eccentric 2-circuit.  
Then $y_{i,j}=x_{i,j+1}$ for all $(i,j)\in\mathbb{Z}_2^2$, by definition,
and $v\neq z_{i,j}$ for all $(i,j)\in\mathbb{Z}_2^2$, by Lemma~\ref{simplify}.
Call $(z_{1,0},\, v)$ the \emph{eccentric arc}.
Note that $z_{1,0}\not\in \{ x_{1,0}, \, x_{1,1}\}$
which is the set of in-neighbours of $v$ on $S$.   
Hence the eccentric arc is never
an arc of $S$, so that the analogue of Case (Na) 
never arises.  The remaining possibilities are below, called Case (Ea)
and (Eb) (these are similar to cases (Nb) and (Nc) for
normal 2-circuits, respectively).
\begin{enumerate}
\item[(Ea)] Suppose that $(z_{1,0},\, v)\not\in A(\chi^{h} Z_J)$.
Then $z_{1,0} x_{1,1} x_{1,0} v$ forms an alternating 4-cycle
which we process using the switch $\chi^h [z_{1,0} x_{1,1} x_{1,0} v]$, 
extending the canonical path by one step to give
\[ G = Z_0,\cdots, Z_J,\, Z_{J+1}.\]
We call this step the \emph{eccentric switch}.
After performing the eccentric switch we have the 2-circuit
\begin{equation}
\label{Seccentric}
S' = v x_{0,0}\cdots z_{1,0} v x_{1,1}  x_{1,0}  \cdots x_{0,0}  x_{0,1}.
\end{equation}
Indeed, since $z_{1,0}\neq x_{1,0}$ it follows that
$S'$ is a normal 2-circuit, which we can process using the method described
in Section~\ref{s:normal}.
This extends the canonical path as
\[ G = Z_0,\cdots, Z_J,\, Z_{J+1}, \, Z_{J+2} \cdots , Z_{J+1 + k}.
\]
Note that $(z_{1,0},\, v)\not\in A(\chi^{h} Z_{J+1+k})$, so the eccentric arc
has been restored to the same state as in $\chi^h Z_J$, before the processing
of $S$ began.
\item[(Eb)]
Suppose that $(z_{1,0},\, v)\in A(\chi^{h} Z_J)$.
Then $S'$ defined in (\ref{Seccentric}) is a normal 2-circuit which we
first process using the method described in Section~\ref{s:normal}.  
This extends the canonical path as
\[ G = Z_0, \cdots, Z_J,\, Z_{J+1},\cdots, Z_{J+k}.\]
Then $z_{1,0} x_{1,1} x_{1,0} v$ forms an alternating 4-cycle
which we process using the eccentric switch 
$\chi^h [z_{1,0} x_{1,1} x_{1,0} v]$, 
extending the canonical path by one step to give
\[ G = Z_0, \cdots, Z_J,\, Z_{J+1},\cdots, Z_{J+k},\, Z_{J+k+1}.\]
Now $(z_{1,0},\, v) \in A(\chi^{h} Z_{J+1+k})$, so the eccentric arc has been 
restored to the same state as in $\chi^h Z_J$.
\end{enumerate}

This procedure for still works even for eccentric
2-circuits with only five vertices.  These arise when $z_{i,j}=x_{i+1,j+1}$
for all $(i,j)\in\mathbb{Z}_2^2$ (matching Figure~\ref{eccentric} with
both curved lines replaced by one arc each.)

The following information will be needed when analysing the flow.

\begin{lemma}
\label{eccentric-plus-shortcut}
Let $S$ be an eccentric $2$-circuit 
with the labelling of (\ref{Slabels})
and let $S'$ be the normal $2$-circuit used to process $S$.   
Suppose that $S'$ falls into case (Nb) or (Nc).
Then the following all hold:
\begin{enumerate}
\item[\emph{(i)}] 
Neither of the arcs $(v,x_{0,1})$, $(x_{1,1},v)$ 
are involved in the eccentric switch.
\item[\emph{(ii)}]
Using the labelling from Figure~\ref{eccentric}, the shortcut arc 
used to process $S'$ is $(z_{1,0},x_{1,0})$ and the shortcut
switch is $[x_{1,1} x_{1,0} z_{1,0} v]$. 
\item[\emph{(iii)}] The
eccentric arc is involved in the shortcut switch and does not lie on the
1-circuit used to process $S'$.  
\end{enumerate}
\end{lemma}

\begin{proof}
Recall that the eccentric arc is $(z_{1,0},v)$.
The first statement is immediate as the eccentric switch processes
the alternating 4-cycle $z_{1,0} x_{1,1} x_{1,0} v$.

For the remainder of the proof, we use 
labels $\hat{x}_{i,j}, \hat{y}_{i,j},\ldots$
to denote the labelling of $S'$ obtained as in (\ref{Slabels}).
See Figure~\ref{eccentric-detail}.
As $S$ is eccentric we have $y_{i,j+1}=x_{i,j}$ for all 
$(i,j)\in\mathbb{Z}_2^2$.
By choice of the eccentric switch we have $\hat{x}_{i,j}= x_{i,j}$ and 
$\hat{y}_{i,j}= y_{i,j}$ for $(i,j)\neq (1,0)$, while
$\hat{x}_{1,0} = z_{1,0}$.
\begin{center}
\begin{figure}[ht]
\psfrag{v}[c]{$v$}
 \psfrag{x00}{$x_{0,0}$}\psfrag{x01}{$x_{0,1}$}\psfrag{x10}{$x_{1,0}$}
 \psfrag{x11}{$x_{1,1}$} \psfrag{a00}{$z_{0,0}$} \psfrag{a01}{$z_{0,1}$}
\psfrag{a10}{$z_{1,0}$} \psfrag{a11}{$z_{1,1}$}
\psfrag{aa10}{$z_{1,0} = \hat{x}_{1,0}$}
\centerline{\includegraphics[scale=0.6]{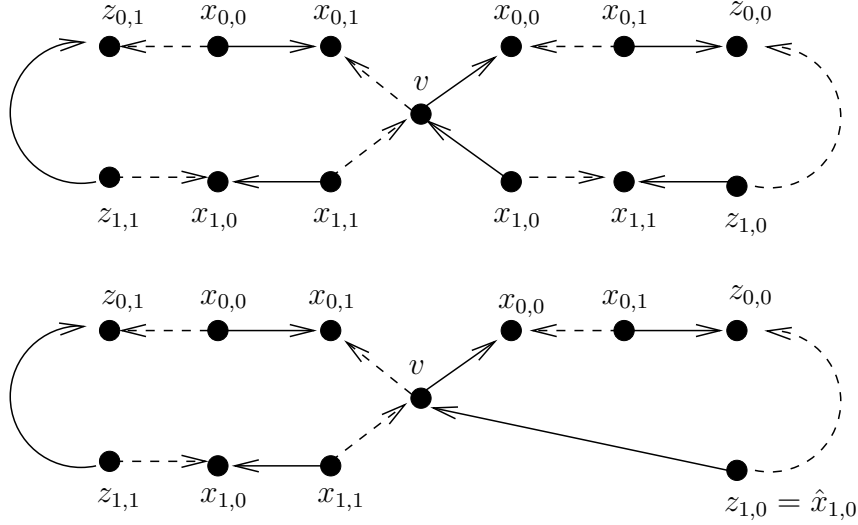}}
\caption{An eccentric 2-circuit $\chi^h S$ (above) and the normal 2-circuit 
$\chi^h S'$ used to process it (below)}
\label{eccentric-detail}
\end{figure}
\end{center}
Now
$z_{1,0}\neq x_{1,0}$ since $z_{1,0} x_{1,1} x_{1,0}$ is a contiguous substring of $S$.
Hence $(1,0)$ is the lexicographically least $(i,j)$ such that
$\hat{x}_{i,j}\neq \hat{y}_{i,j+1}$. 
It follows that the shortcut arc is $(\hat{x}_{1,0},\hat{y}_{1,1})
= (z_{1,0},x_{1,0})$.  Notice that the eccentric arc is incident 
with the shortcut arc at $z_{1,0}$ (with the same orientation).  
Furthermore, the shortcut switch involves
a switch to the alternating 4-cycle
\[ v \hat{x}_{1,1} \hat{y}_{1,1} \hat{x}_{1,0} = v x_{1,1} x
_{1,0} z_{1,0} \]
which includes the eccentric arc.
Specifically, the switch is $[x_{1,1} x_{1,0} z_{1,0} v]$,  
proving (ii).
Since the eccentric arc $(z_{1,0},v)$ is one of the arcs involved
in the shortcut switch, 
it does not lie on the 1-circuit used to process $S'$. 
This establishes (iii), completing the proof.
\end{proof}

\subsection{Processing a triangle }\label{s:triangle}

Now suppose that $S$ is a triangle, with vertices labelled $v_0,v_1,v_2$
where $v_0$ is the least vertex on $S$ and $(v_0,v_1)$ is an arc in 
the current digraph $Z_J$.
Define the sets $\C^{(i,j)} = 
\C^{(i,j)}(\{v_0,v_1,v_2\},Z_J)$ for
${(i,j)}\in \mathbb{Z}_2^2$. 
There are two cases, depending on whether a useful neighbour of $S$
exists. 
\begin{enumerate}
\item[(T1)]
First suppose that there exists a useful neighbour of $S$.
Let $x$ be the minimum useful neighbour of $S$, and set
$(i,h)$ according to the first condition in this list which is
satisfied by $x$:
\[ (i,h) = \begin{cases} (0,0) & \text{ if $x$ is an out-neighbour of
                                      exactly one vertex of $S$,}\\
                         (0,1) & \text{ if $x$ is an out-neighbour of 
			              exactly two vertices of $S$,}\\
             (1,0) & \text{ if $x$ is an in-neighbour of exactly one
	                           vertex of $S$,}\\
             (1,1) & \text{ if $x$ is an in-neighbour of exactly two
	                       vertices of $S$.}
			      \end{cases} \]
Then the sequence of three switches given by 
LaMar~\cite[left half of Figure 2]{lamar}
can be used to process $S$.  For completeness we describe these switches
here. Relabel the vertices of the triangle with $a$, $b$, $c$ so that
\begin{itemize}
\item $\zeta^i (a,x)\in A(\chi^h Z_J)$,
\item $\zeta^i (b,x)\not\in A(\chi^h Z_J)$, $\zeta^i (c,x)\not\in A(\chi^h Z_J)$,
\item $\zeta^{i}(a,b), \zeta^{i}(b,c), \zeta^{i}(c,a)\in A(\chi^h Z_J)$.
\end{itemize}
(Once $x,i,h$ are chosen using the above procedure, the labelling of 
the triangle is uniquely determined.) 
Then the sequence of switches
\[ \zeta^i \chi^h [axbc],\qquad \zeta^i \chi^h [bxca],\qquad 
           \zeta^i \chi^h [abcx]\]
processes the triangle and restores all arcs between $x$ and the
triangle to their original state.  See Figure~\ref{useful-nb}
for the case $(i,h)=(0,0)$: the diagram for the other cases
can be obtained by reversing all arcs if $i=1$, and/or by exchanging 
solid lines and dashed lines if $h=1$.
Call arcs $\zeta^i (a,x)$, $\zeta^i (b,x)$,
$\zeta^i (c,x)$ the \emph{auxilliary arcs}.
\begin{center}
\begin{figure}[ht]
 \psfrag{v0}{$a$}\psfrag{v1}{$b$}\psfrag{v2}{$c$}
 \psfrag{x}{$x$}
\centerline{\includegraphics[scale=0.6]{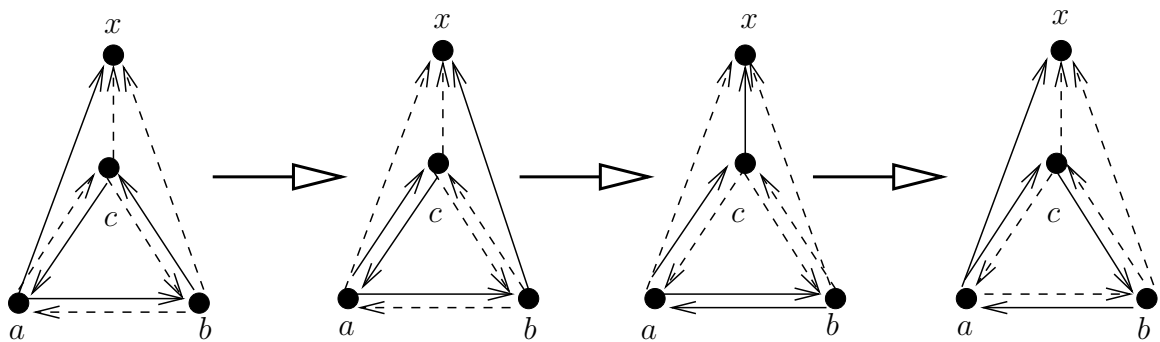}}
\caption{Processing a triangle using a useful neighbour}
\label{useful-nb}
\end{figure}
\end{center}
Use this sequence of 
switches to process the triangle, extending the canonical path as
\[ G = Z_0,\ldots, Z_J, \, Z_{J+1},\, Z_{J+2},\, Z_{J+3}.\]
\item[(T2)]
Suppose that there is no useful neighbour of $S$ in $Z_J$.
Then using Lemma~\ref{useful}, there must exist a useful arc
for $S$.  Let $(x,y)$ be the lexicographically least such arc.
Recall that $(x,y)$ satisfies one of the properties (U1), (U2)
given just before Lemma~\ref{useful}.  Define
\[ h = \begin{cases} 0 & \text{ if (U1) holds,} \\
                     1 & \text{ if (U2) holds.}
\end{cases}
\]
Then $(x,y)\in A(\chi^h\, Z_J)$ with  
$x\in\C^{(h,h)}\cup\C^{(h,h+1)}$ and
$y\in\C^{(h,h)}\cup\C^{(h+1,h)}$. 
Relabel the vertices of the triangle as $a$, $b$, $c$, where
$a=v_0$ and $(a,b)\in A(\chi^h Z_J)$.  (Once $h$ is defined,
this labelling is completely determined.)
The sequence of switches given by LaMar~\cite[right side of Figure 2]{lamar}
will be used to process $S$.  For completeness we give this
sequence of switches in our notation:
\[ \chi^h [x y a b], \quad  \chi^h [a y b c], \quad
   \chi^h [b y c a],\quad \chi^h [x b c y].\]
These switches are also displayed in Figure~\ref{useful-arc} in
the case that $h=0$:  the diagram for $h=1$ can be obtained by
exchanging solid lines and dashed lines.
The arcs
$(x,y)$, $(x,b)$, $(a,y)$, $(b,y)$, $(c,y)$ are called
\emph{auxilliary arcs}. 
Use this sequence of switches to process the triangle, extending
the canonical path as
\[ G = Z_0,\ldots, Z_J, \, Z_{J+1}, \, Z_{J+2}, \, Z_{J+3}, \, Z_{J+4}.\]
\begin{center}
\begin{figure}[ht]
 \psfrag{v0}{$a$}\psfrag{v1}{$b$}\psfrag{v2}{$c$}
 \psfrag{x}{$x$}\psfrag{y}{$y$}
\centerline{\includegraphics[scale=0.6]{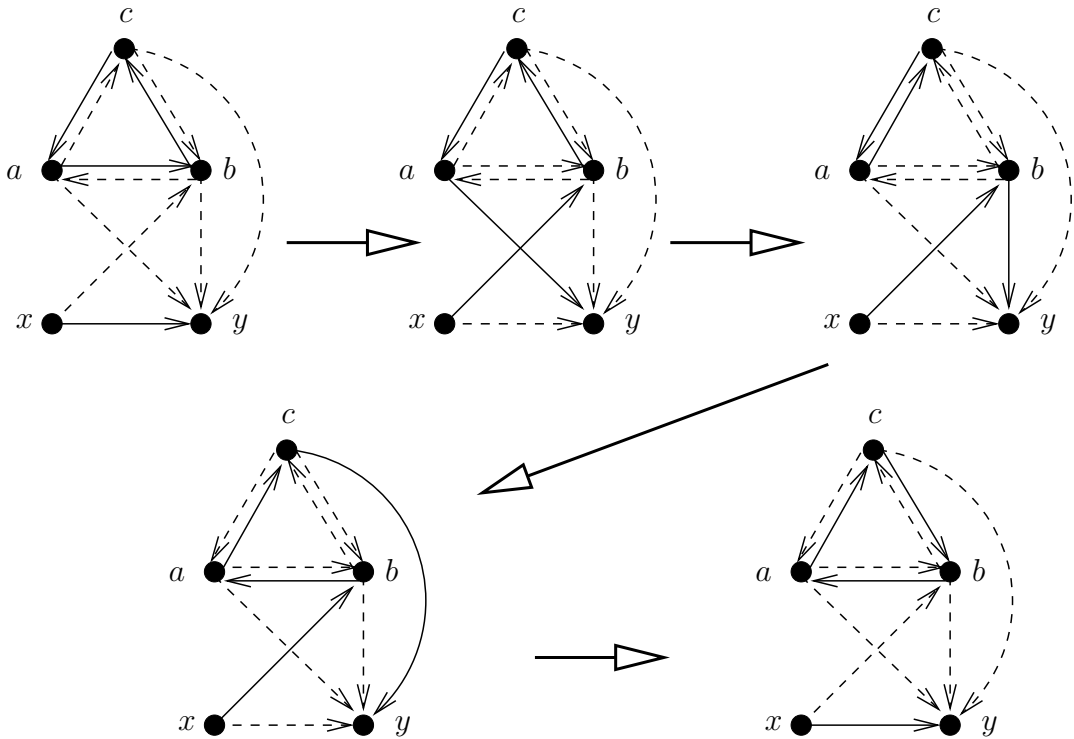}}
\caption{Processing a triangle using a useful arc}
\label{useful-arc}
\end{figure}
\end{center}
\end{enumerate}

\section{Analysing the flow}\label{s:analysis}

We now analyse the multicommodity flow so that we can apply
Lemma~\ref{second-eigval} to give a bound on the second-largest
eigenvalue of the switch chain.
In this section we assume that $1\leq d = d(n)\leq n/2$ for all $n$.
This implies the general result for any $(d(n))$, by complementation
where necessary.  

Fix a pairing $\psi\in\Psi(G,G')$ and let $\gamma_\psi(G,G')$ be
the canonical path from $G$ to $G'$ with respect to $\psi$.
Let $(Z,Z')$ be any transition on $\gamma_\psi(G,G')$, and 
let $S$ be the raw 1-circuit 
or raw 2-circuit which is currently being processed.  
(That is, the transition $(Z,Z')$ is performed while processing $S$.)
Let $Z_J$ be the digraph on the canonical
path from $G$ to $G'$ just before the processing of $S$ began.
Any arc which does not belong to $S$ but which has distinct status
in $Z$ and $Z_J$ is called an \emph{interesting arc}
for $Z$ with respect to $(G,G',\psi)$.
(That is, the arc does not belong to $S$ but is present in $Z$ but
absent in $Z_J$, or vice-versa.)
The only arcs that can be interesting are:
\begin{itemize}
\item odd chords which are switched while processing a 1-circuit,
\item the shortcut arc and/or eccentric arc, switched while processing a 
normal or eccentric 2-circuit,
\item auxilliary arcs which are switched while processing a triangle.
\end{itemize}
We will label an interesting arc by $-1$ (respectively, 2)
if it is absent (respectively, present) in $Z_J$
but present (respectively, absent) in $Z$.   (The reason for this
choice of labels will be made clear shortly.)

Interesting arcs play a key role in our analysis.
The following lemma describes the possible subdigraphs of $Z$ that
can be formed by interesting arcs in $Z$.
It proves that the labelled digraph consisting of the interesting arcs
is a subdigraph of one of the eight labelled digraphs shown in 
Figure~\ref{f:zoo}, up to symmetries.
Here $\{ \mu, \nu\} = \{ -1, 2\}$ and $\{ \xi, \omega\} = \{ -1, 2\}$
independently, giving four symmetries obtained by exchanging these pairs.
Furthermore, $\zeta$ may also be applied to reverse the orientation
of all arcs.  Hence each digraph shown in Figure~\ref{f:zoo} represents
up to eight possible digraphs.  Note, the label for a given arc is shown
next to the head of that arc.
\begin{center}
\begin{figure}[ht]
\psfrag{a}{$\mu$}\psfrag{b}{$\nu$}
\psfrag{z}{$\xi$} \psfrag{x}{$\omega$}
\centerline{\includegraphics[scale=0.5]{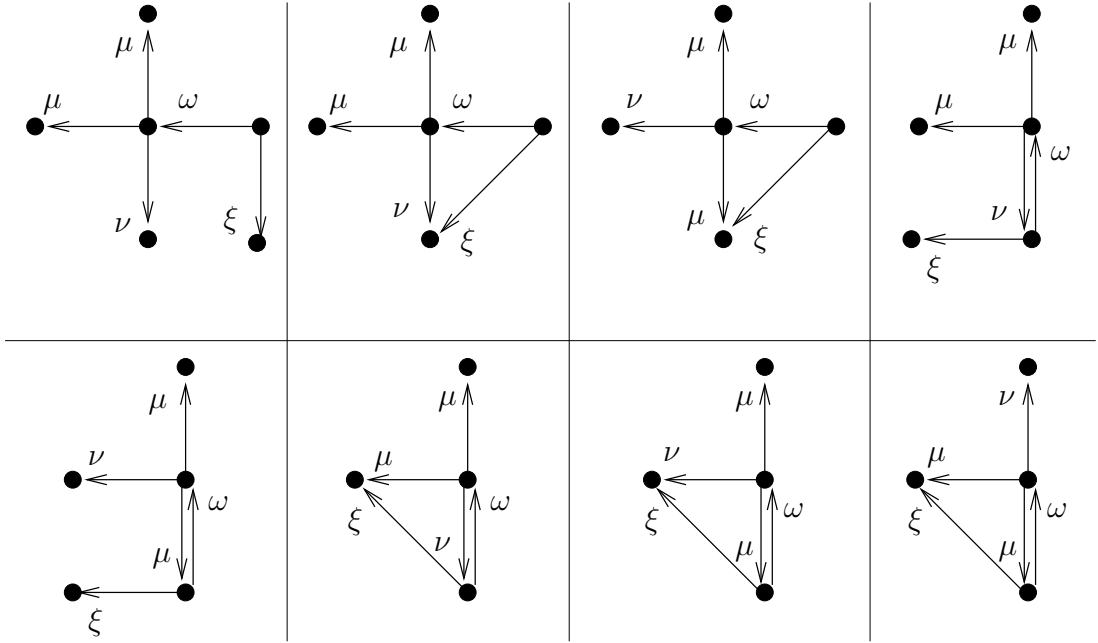}}
\caption{Possible configurations of interesting arcs, up to symmetries}
\label{f:zoo}
\end{figure}
\end{center}

\begin{lemma}
Let $Z$ be a digraph which lies on the canonical path from $G$ to $G'$
with respect to the pairing $\psi\in\Psi(G,G')$.
There are at most five interesting arcs in $Z$ with respect to
$(G,G',\psi)$.  The digraph
consisting of the interesting arcs in $Z$ 
is a subdigraph of one of the digraphs in Figure~\ref{f:zoo}.
If there are five interesting arcs then the following statements all hold:
\begin{enumerate}
\item[\emph{(i)}] There exists a vertex $w$ which
is the head (respectively, tail) of three interesting arcs,
and these three
interesting arcs do not all have the same label.  
\item[\emph{(ii)}] There is a fourth interesting arc which has $w$ as tail 
(respectively, head).  Let $u$ be the head (respectively, tail) of the fourth 
interesting arc.
\item[\emph{(iii)}] The fifth interesting arc is not incident with $w$ but
has $u$ as its head (respectively, tail).
\end{enumerate}
\label{zoo}
\end{lemma}

\begin{proof}
While processing a triangle, at most three interesting arcs are used, 
namely the two or three auxilliary arcs.
It follows from Figures~\ref{useful-nb},~\ref{useful-arc}
that the auxilliary arcs always form subdigraphs of a configuration
from Figure~\ref{zoo}.  

When processing a normal 2-circuit, the situation is very similar
to that in~\cite{CDG}, with at most four interesting arcs.   Up to three
interesting arcs arise from the processing of a 1-circuit.  
They are all odd chords, and hence are all incident
with the start-vertex of the 1-circuit with consistent orientation.
However, they do not all have the
same label.  The fourth interesting arc corresponds to the shortcut arc,
which may be labelled $-1$ or 2 and may be incident with none, one or two 
of the other interesting arcs (but not incident
with the start vertex of the 1-circuit).  

The fifth possible interesting arc is the eccentric arc, in the case that we
are processing an eccentric 2-circuit $S$.  Let $S'$ be the normal 2-circuit
containing the eccentric arc which is used to process $S$.
If $S'$ falls into case (Na) then the eccentric arc may be a 
interesting arc for part (either the start or end) of the processing of $S_1$, 
the 1-circuit which contains it.    But the configuration of interesting
arcs in
this case looks just the same as those which may arise from the processing
of a normal 1-circuit, since the eccentric arc is involved in either the first
switch of the last phase or the last switch of the first phase, and
hence plays the same role as an interesting arc left over from a previous phase.
However, if $S'$ falls into case (Nb) or (Nc) then
by Lemma~\ref{eccentric-plus-shortcut} (iii),
the eccentric arc does not lie on the 1-circuit $S_1$ which arises from $S'$.
But the eccentric arc be an interesting arc throughout the processing of $S_1$.
Hence $S_1$ may have up to five interesting arcs, including the 
shortcut arc and
the eccentric arc.  In this case the eccentric arc is incident with the 
start-vertex $v$
of $S_1$ (which equals the start-vertex of $S$) and it has the \emph{opposite}
orientation to the other interesting arcs incident with $v$, if any.  
Let $u$ be the endvertex of the eccentric arc which is not $v$.
If the shortcut arc is present then it must be incident
with the eccentric arc at $u$, with consistent orientation.
This completes the proof.
\end{proof}

Now identify a digraph with its $n\times n$ adjacency matrix (which has
zero diagonal),  and define the $n\times n$ matrix $L$ by $L+Z=G+G'$.
Entries of $L$ belong to $\{ -1,\, 0,\, 1,\, 2\}$.  We may also think of
$L$ as the complete digraph on $[n]$
with each arc labelled by the corresponding entry of $L$.
An arc in $L$ is called \emph{bad} if its label is $-1$ or 2.
Note that $L$ is independent of $\psi$.  Call $L$ an \emph{encoding}
for $Z$ with respect to ($G,G')$.  Note that an arc receives label
$-1$ if it is absent in both $G$ and $G'$ but is present in $Z$,
while an arc receives label 2 if it is present in both $G$ and $G'$ but
is absent in $Z$.  Thus arcs in the symmetric difference $G\triangle G'$
are never bad arcs.  Furthermore, every bad arc is an interesting arc,
and an interesting arc is bad if and only if it does not belong to
the symmetric difference $G\triangle G'$.  This observation will be
used many times in our analysis.  In particular, it means that the
digraph of bad arcs in an encoding $L$ for $Z$ is a subdigraph of
one of the digraphs in Figure~\ref{f:zoo}.  This also explains
our choice of labels for interesting arcs, since a bad arc
with label 2 (respectively, $-1$) is also an interesting arc with label 2 
(respectively, $-1$).

In the undirected setting~\cite[Lemma 1]{CDG} it is always possible to
uniquely recover $(G,G')$ if $(Z,Z')$, $L$ and $\psi$ are known.
We prove a slightly weaker result in the directed setting. 

\begin{lemma}
Given $(Z,Z')$, $L$, $\psi$, there are at most four possibilities for
$(G,G')$ such that $(Z,Z')$ is a transition along the canonical path
from $G$ to $G'$ corresponding to $\psi$ and $L$ is an encoding for
$Z$ with respect to $(G,G')$. 
\label{notquiteunique}
\end{lemma}

\begin{proof}
The matrix $G+G'$ equals $Z+L$. 
From this matrix we can identify all arcs which are present in both
$G$ and $G'$ (entries with value 2 in $G+G'$) and all arcs which are absent
in both $G$ and $G'$ (entries with value 0 in $G+G'$).  We can also identify 
the symmetric difference $H=G\triangle G'$, corresponding to entries 
with value 1 in $G+G'$.  
It remains to assign colours blue and red to the arcs
of $H$ so that blue arcs come from $G$ and red arcs come from $G'$.

From the uncoloured version of $H$ together with $\psi$ we can construct
the circuit decomposition $\mathcal{C}$.
Let $\mathcal{S}$
be the sequence of raw 1-circuits and raw 2-circuits obtained by
decomposing the circuits in $\mathcal{C}$ in order, as described
in Section~\ref{s:circuit}.  The elements of $\mathcal{S}$ are 
pairwise arc-disjoint and their union is $H$.

Suppose that the transition $(Z,Z')$ deletes the arcs
$(\alpha,\beta)$, $(\delta,\gamma)$ and replaces them with 
$(\alpha,\gamma)$, $(\delta,\beta)$.
Call $(\alpha,\beta), (\alpha,\gamma), (\delta,\beta), (\delta,\gamma)$ 
the \emph{switch arcs}.
We classify transitions along the canonical paths into three types as follows:
\begin{enumerate}
\item[] {\bf Type 1:} the transition is any step in the processing of
a 1-circuit used to process $S\in\mathcal{S}$. 
At least one of the switch arcs belong to $S$.
(This includes the case of a raw 1-circuit, in which case the 1-circuit
equals $S$.)
\item[] {\bf Type 2:}  the transition is a shortcut switch or
an eccentric switch used while processing the normal or eccentric
2-circuit $S\in \mathcal{S}$.
At least two of the switch arcs belong to $S$.
\item[] {\bf Type 3:} the transition is a step in the processing of
a triangle $S\in\mathcal{S}$.  
At least one of the four switch arcs belong to $S$.
\end{enumerate}
In all cases, at least one of the switch arcs belongs to the
element $S\in\mathcal{S}$ currently being processed.  Therefore,
there are at most four possiblities for $S$, namely, at most one
possibility for each switch arc.  (This follows as elements of
$\mathcal{S}$ are pairwise arc-disjoint.)

Now fix one of the (at most four) possibilities for $S$.
We will show that given this choice (or guess) for $S$,
we can uniquely determine $(G,G')$ by colouring the edges of $H$.
Note that if $S$ is a 2-circuit, its labelling
(as in Figure~\ref{2circuit})
can be determined uniquely.   Hence
we can determine whether $S$ is normal, eccentric or a triangle.

Furthermore, in the first two cases we can
identify exactly which arcs will be used as odd chords, shortcut arcs or
eccentric arcs during the processing of $S$.
We now claim that if $S$ is a triangle then we can
uniquely determine the useful neighbour $x$ or the useful arc
$(x,y)$ which is used to process $S$, and hence identify all auxilliary
arcs used while processing $S$.  To see this, note that 
when processing
a triangle, each switch involves either two or three vertices of the
triangle.   If all three vertices of
the triangle are involved in the switch then the other vertex is either
a useful neighbour, or an endvertex of a useful arc.  
Fix one orientation around the triangle and call it ``clockwise'',
with the opposite orientation called ``anticlockwise''.
Consider the number of clockwise
and anticlockwise arcs on the triangle in $Z$ and $Z'$: if they are equal
in $Z$ or in $Z'$ then we are using a useful arc 
and otherwise we are using a useful neighbour.  
(See Figures~\ref{useful-nb},~\ref{useful-arc}.) 
In the latter situation it is easy to identify the useful neighbour $x$: 
it is the only
vertex involved in the switch which does not belong to $S$.
This determines the auxilliary arcs (their orientation matches
the orientation of the switch arcs at $x$).
Now suppose that we are using an useful arc $(x,y)$. 
Then we are in case (T2), which means that no useful neighbour
of $S$ existed at the start of processing $S$.   Then $y$ is the only
vertex incident with the switch arcs which does not belong to $S$,
and $x$ and $y$
are the unique vertices in $Z$ which are useful neighbours of $S$.
That is, $x$ and $y$ are the only vertices not in $S$ which
do not belong to the set $\cup_{(i,j)\in\mathbb{Z}_2^2}
\C^{(i,j)}(S,Z)$.   If only two
vertices of the triangle are involved in the switch then the unique
switch arc which is not incident with either of these vertices is
the useful arc, and the switch is the first or last in processing
$S$.  This shows that
all auxilliary arcs for $S$ can be identified, as claimed.

Suppose that $S$ comes from the decomposition of the circuit 
$C_r\in\mathcal{C}$.   
The digraph induced by all interesting arcs contains no 
circuits, as can be seen from Figure~\ref{f:zoo}.  
Hence for any $\ell\neq r$ we can find at least
one arc on $C_\ell$ which is not an interesting arc
for $S$:  call this a \emph{helpful arc} for $C_\ell$.  
Colour the helpful arc for $C_\ell$ blue if it does not belong
to $Z$ and $\ell < r$, or if it does belong to $Z$ and $\ell>r$; otherwise
colour it red.  Then the colouring of the rest of $C_\ell$ is forced,
since colours alternate around the circuit.
In the same way we can assign colours to the
arcs of every raw 1-circuit and raw 2-circuit obtained in
the decomposition of $C_r$, other than the element $S\in\mathcal{S}$
(which we have assumed is) being switched in the current transition $(Z,Z')$.
It remains to explain how to assign colours to the arcs of $S$.

If $S$ is a triangle then $(Z,Z')$ is a Type 3 transition.
By observing the number of clockwise and anticlockwise arcs in $Z$ and
$Z'$, we can determine the orientation of the triangle in $G$ and in
$G'$ and hence assign colours to the arcs in $S$.  

Hence for the remainder of the proof we can assume that $S$ is
either a 1-circuit or a normal or eccentric
2-circuit.  Therefore the vertices $\alpha, \beta, \gamma, \delta$ 
all belong to $S$ and
without loss of generality $\alpha = \min\{ \alpha,\beta,\gamma,\delta \}$ 
is the start-vertex
of $S$ (since the start-vertex is involved in every switch).
The argument for 1-circuits and normal 2-circuits is very
similar to that given in~\cite{CDG}.

First suppose that $(Z,Z')$ is a Type 1 transition, performed while
processing the 1-circuit $S'$. 
Now  $S'$ may be a raw 1-circuit (in which case $S'=S\in\mathcal{S}$), 
or $S'$ may have arisen while processing a raw (normal or eccentric) 
2-circuit $S$.  Hence $S'$ may contain a shortcut arc
(but note, no 1-circuit contains an eccentric arc, by 
Lemma~\ref{eccentric-plus-shortcut}).  The arcs of
$S'$ can be partitioned into sections, separated from each other by
two consecutive arcs that are either both in $Z$ or both absent from $Z$.
Each section contains
at least two arcs, so at least one arc which is not the shortcut arc.
Then at least
one arc of $S'$ is actually switched in the current transition, which allows
us to label the section containing that arc as switched, and alternately
label the remaining sections around $S'$ as switched or unswitched.
Then colour an arc of $S'$ blue if it belongs to $Z$ and is unswitched or it is
absent from $Z$ and is switched, and colour an arc of $S'$ red if it belongs
to $Z$ and is switched or it is absent from $Z$ and is unswitched.
Finally, if $S'$ is not raw but arose from a 2-circuit $S$, there is
a unique way to colour the remaining arcs of $S$, keeping the
colours alternating.

For the remainder of the proof we assume that $(Z,Z')$ is Type 2 transition 
for $S$; that is, a shortcut switch or an eccentric switch.  
Let $Z_J$ denote the digraph on the canonical path from $G$ to $G'$
just before we start decomposing $S$.
We consider three subcases.

Firstly, suppose that $(Z,Z')$ is an eccentric switch. Then we know that
the arcs $(v,x_{0,1})$ and $(x_{1,0},x_{1,1})$ have the same
status in $Z_J$.  The former is not involved in the eccentric switch,
by Lemma~\ref{eccentric-plus-shortcut} (i), while the latter is involved
in the eccentric switch.  Hence if these two arcs have matching
status in $Z$ then we are in 
Case (Ea) and the current transition is the first in processing $S$.
Colour the arcs of $S$ according to $Z$:  arcs of $S\cap Z$ 
should be coloured blue and the remaining arcs of
$S$ should be coloured red.
If these arcs have opposite status in $Z$ then we are in case (Eb)
and the current transition is the last in processing $S$.
Colour the arcs of $S$ according to $Z'$:  arcs of $S\cap Z'$
should be coloured red and the remaining arcs of $S$ should be
coloured blue.

We proceed similarly if $(Z,Z')$ is a Type 2 transition for $S$ which is
a shortcut switch.   For now, assume that $S$ is a normal 2-circuit,
so the shortcut switch does not involve the eccentric arc (if any).
If the shortcut arc
is $\zeta^i (y_{i,j+1},x_{i,j})$ then the arcs $\zeta^i(x_{i+1,j},v)$
and $\zeta^i (v,x_{i,j})$ have matching status 
in $Z_J$.  The former arc is not involved in the shortcut switch but the latter arc is.
Hence if these two arcs still have matching status in $Z$ then we are in 
Case (Nb) and the current transition is the first in processing $S$.
Colour the arcs of $S$ according to $Z$, as described in the 
previous
paragraph.  If these two arcs have opposite status in $Z$ (one absent
and one present)
then we have already processed the 1-circuit using the shortcut arc,
so we are in case (Nc) and the current transition is the last in
processing $S$.  Colour the arcs of $S$ according to $Z'$, as
described in the previous paragraph.

The third subcase is that $(Z,Z')$ is a shortcut switch 
which also involves an eccentric arc.  Then $S$ is an eccentric 2-circuit
which has been decomposed into an eccentric switch and a normal 
2-circuit $S'$, where $S'$ contains the eccentric arc.  The current
transition is the shortcut switch which has arisen while processing
$S'$.    
Now the arcs $(v, x_{0,1})$ and $(x_{1,1},v)$
have matching status in $Z_J$. 
From Lemma~\ref{eccentric-plus-shortcut} (i) we know that
neither of these arcs are involved in the eccentric switch.
The former arc is not involved in the shortcut switch but the latter arc is,
by Lemma~\ref{eccentric-plus-shortcut} (ii).
Hence, these two arcs also have matching status at the start of processing
$S'$, and we 
can colour the arcs of $S$ according to $Z$ if these arcs have matching
status in $Z$, and colour the arcs of $S$ according to $Z'$ otherwise.  
This completes the proof.
\end{proof}

Let $L(\alpha,\beta)$ denote the label of arc $(\alpha,\beta)$ 
in the encoding $L$.  
The arc-reversal operator $\zeta$ acts on an 
encoding $L$ by mapping $L$ to its transpose $\zeta L = L^T$.  
If $\zeta^i L(\alpha,\beta)=2$ and $\zeta^i L(\alpha,\gamma)=-1$ for some
$i\in \{0,1\}$ then $(i,\alpha,\beta,\gamma)$ is called a
\emph{handy tuple} with \emph{centre}
$\alpha$.  If $(i,\alpha,\beta,\gamma)$ is handy and at 
most one of $\beta$, $\gamma$ is the head, when $i=0$ (respectively, the 
tail, when $i=1$)
of two bad arcs with distinct labels then $(i,\alpha,\beta,\gamma)$ is
said to be \emph{very handy}.
We now collect together some structural information about bad arcs in encodings.

\begin{lemma}
\label{structure}
Given $G, G'\in \Omega_{n,d}$ with symmetric difference $H = G\triangle G'$,
suppose that $Z$ is a digraph on the canonical path from $G$ to $G'$
with respect to some pairing $\psi$.  Let $L$ be the corresponding
encoding, defined by $L+Z=G+G'$.
Then the following statements all hold.
\begin{enumerate}
\item[\emph{(i)}]  Viewed as the arcs of a labelled digraph, the set of 
bad arcs in $L$ forms a subdigraph of one of the digraphs given in 
Figure~\ref{f:zoo}.
\item[\emph{(ii)}] If $L$ contains a handy tuple then $L$ contains 
a very handy tuple.
\item[\emph{(iii)}] If there are five bad
arcs in $L$ then there exists a very handy tuple $(i_1,\alpha_1,\beta_1,\gamma_1)$, 
and a handy tuple
$(i_1,\alpha_2,\beta_2,\gamma_2)$ in $L$ such that $\alpha_1\neq \alpha_2$
and
\begin{equation}
\label{independent}
\{ \zeta^{i_1} (\alpha_1,\beta_1),\, \zeta^{i_1} (\alpha_1,\gamma_1)\}
   \cap \{ \zeta^{i_2} (\alpha_2,\beta_2),\, \zeta^{i_2} (\alpha_2,\gamma_2) \} 
               = \emptyset.
\end{equation}
\item[\emph{(iv)}] If there are four bad arcs in $L$ then there is at 
                             least one handy tuple in $L$.
\item[\emph{(v)}] If $d=1$ then 
no arc in $H$ is incident with a bad arc with label $2$.
If $L$ has a bad arc with label $2$ which is not the only bad arc in $L$
then $L$ has a handy tuple and $L$ has at most three bad arcs, 
exactly one of which has label $2$.
Furthermore, if $L$ has three bad arcs, exactly one of which has label
$2$ then each endvertex of the bad arc with label $2$ is the centre of
a handy tuple in $L$. 
\item[\emph{(vi)}] If $d=2$ then every vertex which has nonzero degree in
$H$ is the head of at most one bad arc with label $2$ and is the
tail of at most one bad arc with label $2$.
\end{enumerate}
\end{lemma}

\begin{proof}
Every bad arc is interesting, so (i) follows immediately from Lemma~\ref{zoo}.
Statements (ii)--(iv) follow from (i), by inspection of Figure~\ref{f:zoo}.
Now the head (respectively, tail) of a bad arc is also the head
(respectively, tail) of an arc in $G-G'$ and an arc in $G'-G$, 
unless the bad arc is the useful arc used to process a triangle in case
(T2).  This follows from the definition of odd chords, shortcut arc,
eccentric arc and auxilliary arcs in 
Sections~\ref{s:1circuit},~\ref{s:normal}--~\ref{s:triangle}.
Hence (vi) and the first statement of (v) holds,
since a bad arc with label 2 is present in both $G$ and $G'$.
Furthermore, inspection of Figure~\ref{useful-arc}
shows that the remaining statements of (v) hold, completing the proof.
\end{proof}

The notion of an encoding is now generalised to mean any $n\times n$
matrix $L$ with entries in $\{ -1, 0, 1, 2\}$ such that every row
and column sum equals $d$.
Given $Z\in \Omega_{n,d}$, an encoding $L$ is called
\emph{$Z$-valid} if every entry of $L+Z$ 
belongs to $\{ 0,1,2\}$ and $L,Z,H$ satisfy
statements (i)--(vi) of Lemma~\ref{structure}, where $H$ 
is the digraph defined by the entries of $L+Z$ which equal 1.  
We also define the set $\W(L)$ of all bad arcs in $L$ by
\[ \W(L) = \{ (i,j)\in [n]^2 \mid L(i,j)\in \{ -1,2\}\}.\]

\begin{lemma}
Let $Z\in\Omega_{n,d}$ and let $L$ be a $Z$-valid encoding. 
Suppose that $L'$ is another encoding such that
$\W(L')\subseteq \W(L)$.  Then $L'$ is also $Z$-valid.
\label{Zvalid}
\end{lemma}

\begin{proof}
If $L'(i,j) = -1$ then $L(i,j)=-1$ and hence $Z(i,j)=1$, as $L$ is $Z$-valid.
Similarly, if $L'(i,j)=2$ then $L(i,j)=2$ and hence $Z(i,j)=0$.
This shows that every entry of $L'+Z$ belongs to $\{ 0,1,2\}$.
Checking properties (i)--(vi) of Lemma~\ref{structure} we see that
they all hold for $L',Z,H'$, completing the proof.
\end{proof}

Switches can be applied to encodings, as follows.  
By definition, the sum of
all labels on arcs with head $v$ add up to $d$, and the sum of all labels
on arcs with tail $v$ add up to $d$, for all vertices $v$.
If $x,y,z,w$ are vertices with $L(x,y) > -1$, $L(w,z) > -1$,
$L(x,z) < 2$ and $L(w,y) < 2$ then we may perform the switch
$[xywz]$ by decreasing
$L(x,y)$ and $L(w,z)$ by one and increasing $L(x,z)$ and $L(w,y)$ by one,
giving a new encoding $L'$.  

\begin{lemma}
Let $Z\in\Omega_{n,d}$. 
Given a $Z$-valid encoding, one can obtain a digraph (with no bad arcs) using
at most three switches.
\label{fix}
\end{lemma}

\begin{proof}
Let $L$ be a $Z$-valid encoding and let $H$ be the digraph given by
the entries of $L+Z$ which equal 1. 

First suppose that $L$ contains a handy tuple.  Then $L$ contains
a very handy tuple, by Lemma~\ref{structure} (ii). 
If $L$ contains a very handy tuple $(i_1,\alpha_1,\beta_1,\gamma_1)$
and a handy tuple $(i_2,\alpha_2,\beta_2,\gamma_2)$ such that 
$\alpha_1\neq \alpha_2$
and (\ref{independent}) holds, then we choose
$(i,\alpha,\beta,\gamma)$ to be the very handy tuple 
$(i_1,\alpha_1,\beta_1,\gamma_1)$.
Otherwise, let $(i,\alpha,\beta,\gamma)$ be any
very handy tuple in $L$.

If $i=0$ (respectively, $i=1$) then the sum of the labels
on the bad arcs with $\beta$ as head (respectively, tail) is strictly greater
than the sum of the labels on the bad arcs with $\gamma$ as head (respectively,
tail).
By construction, each row of $L$ adds up to $d$ and each column of $L$ adds up to $d$.
Hence $\gamma$ is the head (respectively, tail)
of strictly more good arcs (with label 1) than $\beta$.
It follows that there exists a vertex $\delta$ such that
$\zeta^i L(\delta,\gamma) = 1$ and $\zeta^i L(\delta, \beta)= 0$.
Now we can perform the switch $\zeta^i [\alpha\beta\delta\gamma]$ 
to give an encoding $L'$ with 
$\zeta^i L'(\alpha,\beta)=\zeta^i L'(\delta,\beta)=1$
and $\zeta^i L'(\alpha,\gamma)=\zeta^i L'(\delta,\gamma)=0$.  
Note that $L'$ is a $Z$-valid encoding by Lemma~\ref{Zvalid}, and that
$|\W(L')|=|\W(L)|-2$. 
We call this operation a
$(-1,2)$-\emph{switch}.

Next suppose that no vertex is the head (respectively, tail) of two bad
arcs with distinct labels, but that an arc exists in $L$ with label 2.
By Lemma~\ref{structure} (iii), (iv), there are at most three bad arcs in $L$.
Choose vertices $\alpha,\beta$ such that for some $i\in\{ 0,1\}$
we have $\zeta^i L(\alpha,\beta)=2$
and if $i=0$ (respectively, $i=1$) then $\alpha$ is the tail
(respectively, head) of exactly one bad arc.
(That such an $\alpha$ exists follows from Lemma~\ref{structure} (i).)
Let $U$ be the set of vertices $x\neq \alpha$ with $\zeta^i L(\alpha,x)=0$.
Since $\alpha$ is the tail (respectively, head) of exactly $d-2$ arcs 
labelled 1 and one arc labelled 2, it follows that $|U| = n-d\geq d$.

We claim that there exists a vertex $\gamma\in U$ which is not the
head (respectively, tail) of a bad arc with label 2.
If $d\geq 3$ then there are at most 
3 vertices which are at the head (respectively, tail)
of a bad arc labelled 2, and one of these is $\beta$.  
Since $\beta\not\in U$ and $|U|\geq 3$, we can choose a vertex 
$\gamma\in U$ which is not the head of a bad arc with label 2, as claimed.
If $d=2$  then by Lemma~\ref{structure} (vi), 
each vertex in $H$ is
head (respectively, tail) of at most one bad arc labelled 2.  Hence there are
at most 2 bad arcs labelled 2 in $L$, by Lemma~\ref{structure} (i).
Therefore at most one vertex other than $\beta$ is the
head (respectively, tail) of a bad arc labelled 2 in $L$.
The claim then follows since $\beta\not\in U$ and $|U|\geq 2$. 
If $d=1$ then by Lemma~\ref{structure} (v) there is
exactly one bad arc in $L$, namely $\zeta^i (\alpha,\beta)$.
Hence we can let $\gamma$ be any element of $U$ since $\beta\not\in U$,
and the claim follows as $|U|\geq 1$ in this case.

Now $\beta$ is the head (respectively, tail) of at most $d-2$ good arcs
and $\gamma$ is the head (respectively, tail) of at least $d$ good arcs.  
Hence we can choose a vertex $\delta$ such that $\zeta^i L(\delta,\beta)=0$
and $\zeta^i L(\delta,\gamma)=1$.  Perform the switch
$\zeta^i[\alpha \beta\delta\gamma]$ to produce an 
encoding $L'$  with $\zeta^i L'(\alpha,\beta)=\zeta^i L'(\alpha,\gamma)
=\zeta^i L'(\delta,\beta)=1$ and $\zeta^i L'(\delta,\gamma)=0$.  
Then $L'$ is $Z$-valid by Lemma~\ref{Zvalid}, and
$|\W(L')| = |\W(L)| - 1$. 
Call this operation a 2-\emph{switch}.

Finally, suppose that the only remaining bad arcs are labelled $-1$.
Let $\alpha$ and $\gamma$ be vertices such that $\zeta^i L(\alpha,\gamma)=-1$
for some $i\in \{0,1\}$,  
choosing $\alpha$ to be a vertex at the tail, when $i=0$ 
(respectively head, when $i=1$) of
two bad arcs with label $-1$, if such a vertex exists.
Note that when $L$ contains three bad arcs with label $-1$ then
such a choice of $\alpha$ exists, by Lemma~\ref{structure} (i).
We claim that there exists a vertex $\beta$ such that 
$\zeta^i L(\alpha,\beta)=1$ but $\beta$ is not the head, when $i=0$
(respectively tail, when $i=1$) of any bad arc.  
To see this, note that there are at least $d+1\geq 2$ choices for $\beta$,
and there is at most one vertex which is at the head (respectively, tail)
of a bad arc which is not incident with $\alpha$, by choice
of $\alpha$.  Hence we can avoid this vertex when choosing $\beta$,
proving the claim.
Then $\beta$ is the head (respectively, tail) of exactly $d$ good arcs, 
while $\gamma$ is the head (respectively, tail) of at least $d+1$ good arcs.  
Hence there
is at least one way to choose a vertex $\delta$ such that
$\zeta^i L(\delta,\beta)=0$ and $\zeta^i L(\delta,\gamma)=1$.  Perform the
switch  $\zeta^i [\alpha\beta\delta\gamma]$
to produce an encoding $L'$, with $\zeta^i L'(\alpha,\beta) = 
\zeta^i L'(\delta,\gamma) = \zeta^i L'(\alpha,\gamma)=0$,
$\zeta^i L'(\delta,\beta)=1$.
Again Lemma~\ref{Zvalid} shows that $L'$ is $Z$-valid, and
$|\W(L')| = |\W(L)| - 1$. 
Call this operation a $(-1)$-\emph{switch}.

If the original encoding $L$ has five bad arcs then by 
Lemma~\ref{structure} (iii),
we can find a very handy tuple $(i_1,\alpha_1,\beta_1,\gamma_1)$
in $L$ and perform the $(-1,2)$-switch $\zeta^{i_1} [\alpha_1\beta_1\gamma_1\delta_1]$,
where $\delta_1$ is a vertex found using the procedure above.
It follows from (\ref{independent}) and Figure~\ref{structure} (i)
that $(i_2,\alpha_2,\beta_2,\gamma_2)$
is a very handy tuple in the resulting $Z$-valid encoding $L'$.  Hence we may
perform the $(-1,2)$-switch $\zeta^{i_2} [\alpha_2\beta_2\gamma_2\delta_2]$
to transform $L'$ into the $Z$-valid encoding $L''$ with at most one bad arc.
At most one further switch is required to transform $L''$ into an encoding
with no bad arcs.  Thus at most 3 switches are needed to process $L$ 
when $L$ has five bad arcs.

Similarly, if $L$ has four bad arcs then by Lemma~\ref{structure} (iv), 
we can transform $L$ into a $Z$-valid encoding $L'$ with at
most two bad arcs, using a $(-1,2)$-switch.  At most two further switches
are needed to produce an encoding with no bad arcs.  Thus at most 3 switches
are needed to process $L$ when $L$ has four bad arcs.  Clearly, 
if $L$ has at most
3 bad arcs then at most 3 switches are required.  This completes the proof.
\end{proof}

For $Z\in\Omega_{n,d}$ let $\mathcal{L}(Z)$ be the set of all 
$Z$-valid encodings. 
We obtain the following upper bound on $|\mathcal{L}(Z)|$ using
a relatively simple proof.   It is possible that an improved bound
can be found using a more careful analysis, probably saving a
factor of $n$.

\begin{lemma}
For any $Z\in\Omega_{n,d}$ we have
\[ |\mathcal{L}(Z)| \leq  25\, d^6 n^6\, |\Omega_{n,d}|.\]
\label{poly}
\end{lemma}

\begin{proof}
Fix $Z\in\Omega_{n,d}$ and let $L\in\mathcal{L}(Z)$ be a $Z$-valid encoding.
By Lemma~\ref{fix} there exists a sequence
\[ L = L_0, L_1, \ldots, L_r = A\]
where $A\in \Omega_{n,d}$ is a digraph with no bad arcs, $r\leq 3$
and each of
$L_1,\ldots, L_r$ is $Z$-valid. 
We can turn this into a function 
$\varphi:\mathcal{L}(Z)\to\Omega_{n,d}$ by performing these switches
in a canonical way:  as in Lemma~\ref{fix} perform all $(-1,2)$-switches first,
then all 2-switches, then all $(-1)$-switches, following the extra conditions
described in Lemma~\ref{fix} and breaking ties using lexicographic ordering
on the 5-tuple $(i,\alpha,\beta,\gamma,\delta)$.
It suffices to prove that $|\varphi^{-1}(A)|\leq 25\, d^6\, n^6$
for all $A\in\Omega_{n,d}$.

Now fix $A\in\Omega_{n,d}$.  Define a 
\emph{reverse $X$-switch}
to be the reverse of a $X$-switch, for $X \in \{(-1,2),\, -1,\, 2\}$.
For an upper bound we count all encodings 
which can be obtained from $A$ using at most three reverse switches,
regardless of whether $A$ is the canonical image of that encoding under
$\varphi$.  We will perform the reverse switchings in order:  first 
the reverse $(-1)$-switches, if any, then any reverse 2-switches and 
finally any reverse $(-1,2)$-switches.  
Note that a reverse switching alters four entries of the current encoding, 
none of which are bad entries.  So a bad arc created by
a reverse switch will never be changed by a later reverse switch.

Fix an encoding $B\in\mathcal{L}(Z)$ (which may not have any bad arcs).
Let $N_X(B)$ be the number of distinct 5-tuples 
$(i,\alpha,\beta,\gamma,\delta)$ which define a 
reverse $X$-switch that may be performed in $B$,  
for $X\in\{ (-1,2),\, -1,\, 2\}$.
The result of each of the reverse switches counted by $N_X(B)$ is  
a $Z$-valid encoding.
Our next task is to calculate upper bounds on $N_X(B)$ which hold for
all encodings $B\in\mathcal{L}(Z)$.  

We only perform $(-1)$-switches on encodings $B\in\mathcal{L}(Z)$ which 
have no bad arc with label 2.  For such encodings we claim that
\begin{equation}
\label{N1}
N_{-1}(B)  \leq 2 d^2 n(n-2).
\end{equation}
With notation as in Lemma~\ref{fix}, the factor of 2 counts the two
choices of orientation $i\in \{ 0,1\}$.  
We prove the bound assuming that $i=0$, and the proof for $i=1$
follows by symmetry.  There are $n$ choices for vertex $\alpha$, and 
$d$ choices for
$\gamma$ since $\zeta^i (\alpha,\gamma) \in A(Z)$ as $B$ is $Z$-valid. 
Then choose $\beta\neq \alpha$
so that $B(\alpha,\beta)=0$ and $\beta$ is not the head of any bad arc.
There are at most $n-2$ choices for $\beta$ since 
$\beta\not\in\{ \alpha,\gamma\}$.   
Then there are $d$ choices for $\delta$
such that $B(\delta,\beta)=1$,
since $\beta$ is the head of exactly $d$ good arcs.
This gives the claimed bound on $N_{-1}(B)$ when $B$ has no bad
arcs labelled 2.  

Now suppose that $B\in\mathcal{L}(Z)$ 
may contain bad arcs with distinct labels,  but no vertex is the
head (respectively, tail) of two bad arcs with distinct labels in $B$.
We also ensure that the reverse 2-switchings that we perform 
never create any such pair of bad arcs, in order to maintain the
canonical order in which forward switches are performed.
We claim that
\begin{equation}
\label{N2}
N_2(B)  \leq 2d (d-1)^2 n.  
\end{equation}
The factor of 2 counts the two choices of orientation $i\in \{0,1\}$.
We prove the bound assuming that $i=0$, and the proof for $i=1$
follows by symmetry.
There are at most $n$ choices for $\alpha$ which is not the tail of a 
bad arc labelled $-1$.
Then distinct out-neighbours
$\beta$, $\gamma$ of $\alpha$ in $B$ can be chosen in at most $d(d-1)$ ways
such that $\beta$ is not the head of a bad arc labelled $-1$
and $\gamma$ is not the head of a bad arc labelled 2.
(Note, $\alpha$ is the tail of at most $d$ good arcs, since $\alpha$
is not the tail of any bad arc labelled $-1$.)
Then there are at most $d-1$ choices for
a neighbour $\delta$ of $\beta$ in $B$,
since $\beta$ is the head of at most $d$
good arcs.  This gives the claimed bound on $N_2(B)$. 

Finally, we claim that for all $B\in\mathcal{L}(Z)$ we have
\begin{equation}
\label{N12}
N_{(-1,2)}(B) \leq 2d^2(d+1) n. 
\end{equation}
Again, the factor of 2 counts the two choices of orientation $i\in\{0,1\}$
and we assume $i=0$ below, without loss of generality.
There are $n$ ways to choose a vertex $\alpha$ which may be the
tail of at most one bad arc in $B$.
There are $d$ choices for $\gamma$, as $B$ is $Z$-valid so 
$\zeta^i(\alpha,\gamma)\in A(Z)$.
Then there are at most $d+1$ choices for $\beta$ such that
$\beta$ is an out-neighbour of $\alpha$ and is not the head of any arc
labelled $-1$.  (There are at most $d$
choices for $\beta$ if there is no bad arc incident with $\alpha$ in $B$.)
Finally, there are at most $d$ choices for $\delta\neq\alpha$ 
such that $B(\delta,\beta)=1$, 
since $\beta$ is the head of at most $d$ good arcs.  
(The $d$ here arises since $\beta$ may itself be the head of at most one
bad arc in $B$,
and the bad arc may be labelled $-1$.)  
This gives the claimed bound on $N_{(-1,2)}(B)$.

Each sequence of reverse switches which may arise is given a type,
defined by the corresponding sequence of labels in $\{ -1, 2, (-1,2)\}$.
It follows from the proof of Lemma~\ref{fix} 
that the only types of reverse switchings which occur are given by
the following 9 sequences and all distinct subsequences of these 
(including the empty sequence):
\[
\begin{array}{lll}
 [\ -1,\ (-1,2),\ (-1,2)\ ], & [\ 2,\ (-1,2),\ (-1,2)\ ],  & [\  -1,\ -1,\ (-1,2)\ ], \\{}
 [\ -1,\ 2,\ (-1,2)\ ],  & [ \ 2,\ 2,\ (-1,2)\ ], & [\ -1,\ -1,\ -1\ ],\\{}
 [\ -1,\ -1,\ 2\ ], & [\ -1,\ 2,\ 2\ ], & [\ 2,\ 2,\ 2\ ]. 
\end{array}
\]
This gives 19 possible types  in all.
We calculate the contribution of a type by simply multiplying the upper
bounds obtained in (\ref{N1})--(\ref{N12}) corresponding to each reverse
switch in the sequence.  
(It is at this step that a more careful analysis may lead to an
improved bound, but we are satisfied by the bound given by this simple
calculation.)
For example, the contribution from the type $[\ -1,\ (-1,2),\ (-1,2)]$ is
\[  2d^2 n(n-2) (2d^2(d+1)n)^2  = 8 d^6 (d+1)^2 n^3(n-2).\]

Finally we simply sum the contribution from each of the 19 types
and find that the resulting expression is bounded above by  
$25 d^6 n^6$, using the inequalities $1\leq d\leq n/2$.
This shows that
\[ |\varphi^{-1}(A)| \leq  25\, d^6 n^6, \]
completing the proof.   
\end{proof}

For each pair $(G,G')$ of distinct digraphs in $\Omega_{n,d}$, let
$\mathcal{P}_{G,G'}$ be the set of $|\Psi(G,G')|$ canonical paths
which we have defined from $G$ to $G'$, one for each pairing
$\psi\in\Psi(G,G')$.  Let $\mathcal{P} = \cup_{G\neq G'} \mathcal{P}_{G,G'}$.
Define
\[ f(\gamma) = |\Omega_{n,d}|^{-2}\, |\Psi(G,G')|^{-1}\]
for each path $\gamma\in\mathcal{P}_{G,G'}$.  Then
\[ \sum_{\gamma\in\mathcal{P}_{G,G'}} f(\gamma) = |\Omega_{n,d}|^{-2}
          = \pi(G)\,\pi(G')\]
where $\pi$ is the stationary distribution of the Markov chain,
which is uniform on $\Omega_{n,d}$.  Thus $f:\mathcal{P}\to [0,\infty)$
is a flow.  We want to apply Lemma~\ref{second-eigval}.  First we bound
$f(e)$ for all transitions $e$ of the Markov chain.

\begin{lemma}
For any transition $e=(Z,Z')$ of the Markov chain,
\[ f(e) \leq 100\, d^{22}\, n^6\, |\Omega_{n,d}|^{-1}.\]
\label{load}
\end{lemma}

\begin{proof}
Fix a transition $e=(Z,Z')$ of the Markov chain.
Let $(G,G')$ be a pair of distinct digraphs in $\Omega_{n,d}$
and suppose that $e$ lies on $\gamma_\psi(G,G')$, the canonical
path from $G$ to $G'$ corresponding to the pairing $\psi\in\Psi(G,G')$.
From $Z$ and $(G,G')$ we can construct $L$ and the digraph
$H=Z\triangle L = G\triangle G'$.  We colour arcs of $H$ green
if they belong to $Z$ and yellow if the corresponding entry in
$L$ is 1. 
(Recall that the symmetric
difference $H$ consists of those arcs with entry 1 in $L+Z = G+G'$.)

From the pairing $\psi$ we obtain the circuit decomposition
$\mathcal{C}$ of $H$, with colours alternating green, yellow
almost everywhere.
A vertex $x$ is \emph{bad} with respect
to $\psi$ if two arcs of the same colour are paired at $x$ under $\psi$.
If a vertex is not bad it is called \emph{good}.
Every bad vertex
lies on the circuit currently being processed. Specifically, bad 
vertices may only be found incident to interesting arcs. 
Lemma~\ref{zoo} shows that there are at most
five interesting arcs and at most six potentially bad vertices.

A yellow-yellow
or green-green pair at a bad vertex $x$ is called a \emph{bad pair}
with respect to $\psi$.
Careful consideration of the possibilities reveals that
there can be at most 16 bad pairs with respect to $\psi$.  
In the worst case, there are five interesting arcs which all belong 
to $H$.   An interesting arc $e$ which belongs to $H$ 
creates two bad pairs in the circuit containing $e$, one at each
endvertex of $e$ (both of the same colour).
A bad pair is also created in the current circuit $C$ 
incident with each endvertex of each interesting arc, 
giving at most six further bad pairs.
(The worked example in Section~\ref{a:example} gives an example of
a digraph, $Z_4$, containing the maximum number of bad pairs:
see Figure~\ref{example5}.)

Note also that a bad vertex may be the head (respectively, tail)
of at most two bad pairs of each colour.  This follows from 
Lemma~\ref{zoo} since no vertex is head (respectively, tail) 
of more than two interesting arcs with the same label.  Hence a bad vertex may 
be the head (respectively, tail) of at most four bad pairs in total.
This is true even if there are some coincidences between the bad 
vertices, which may occur when the interesting arcs have one of the
configurations other than the first one in Figure~\ref{f:zoo}.
To see this, note that for all the configurations in Figure~\ref{f:zoo},
the only vertex which is the head (or tail) of more than two interesting
arcs 
is $v$, the start-vertex of the current circuit, and $v$ is always
distinct from all other bad vertices.

Given the uncoloured digraph $H$, we can form a pairing $\psi$ by pairing
up all in-arcs around $v$ and pairing up all out-arcs around $v$,
for each vertex $v$.  Let the set of all these pairings be $\Psi(H)$.
Say that a pairing $\psi\in\Psi(H)$ is \emph{consistent with} $L$
if there are at most 16 bad pairs in the yellow-green colouring of
$H$ with respect to $L$, and at each vertex $u$ 
and for each choice of orientation there are
at most two bad pairs of each colour with that orientation at $u$.   
Let $\Psi'(H,L)$ be the set of all pairings
$\psi$ of $H$ which are consistent with $L$.  Given any $(G,G')$
with $G\triangle G'=H$, any pairing $\psi\in\Psi(G,G')$ is consistent
with the yellow-green colouring of $H$, as proved above.  Therefore each triple
$(G,G',\psi)$ with $\psi\in\Psi(G,G')$ and $e\in\gamma_\psi(G,G')$
gives rise to at least one pair $(L,\psi)$ with $L\in\mathcal{L}(Z)$
and $\psi\in\Psi'(H,L)$.

Conversely, we can start with $L\in\mathcal{L}(Z)$ and find an
upper bound for $|\Psi'(H,L)|$.  Once $\psi$ and $(Z,Z')$ are given,
there are at most four possibilities for $(G,G')$ with
$e\in\gamma_\psi(G,G')$, by Lemma~\ref{notquiteunique}.
Recall from (\ref{number-pairings}) that
\[ |\Psi(G,G')| = \prod_{v\in V} \theta_v!\, \phi_v!\]
where $2\theta_v$ is the in-degree of $v$ in $H$ and $2\phi_v$ is the
out-degree of $v$ in $H$.  
Similarly, each good vertex $v$ contributes a
factor $\theta_v!\, \phi_v!$ to $|\Psi'(H,L)|$, but a bad vertex may
contribute more.   The contributions from in-arcs and out-arcs are
independent, so we consider only in-arcs below.  

Recall that no vertex can be the head of more than two bad pairs of
a given colour.  
First suppose that a vertex $v$ is the head of
$\theta_v+2$ green arcs and $\theta_v-2$ yellow arcs.
Then $v$ must be bad, with two bad green pairs and no bad yellow pairs.
The number of ways to pair up the in-arcs around $v$
is
\[ 3 \, \binom{\theta_v+2}{4}\, (\theta_v-2)! 
       = \frac{(\theta_v + 2)(\theta_v+1)}{8}\, \theta_v!
       \leq \theta_v^2\cdot \theta_v! \leq d^2 \theta_v!.
\]
Next suppose that $v$ is the head of $\theta_v+1$ green arcs 
and $\theta_v-1$ yellow arcs.  Then $v$ must be a bad vertex.
Now $v$ may be the head of two bad green pairs and one bad yellow pair,
or $v$ may be the head of one bad green pair and no bad yellow pairs.
The number of ways to pair up the in-arcs around $v$ with two bad
green pairs and one bad yellow pair is
\[  3\, \binom{\theta_v+1}{4}\,\binom{\theta_v-1}{2}\, (\theta_v-3)! 
 = \frac{(\theta_v+1)(\theta_v-1)(\theta_v-2)}{16}\, \theta_v!
   \leq \theta_v^3\, \theta_v!
  \leq d^3\, \theta_v!, 
\]
while the number of pairings of in-arcs around $v$ with one bad green pair
and no bad yellow pairs is
\[  \binom{\theta_v+1}{2}\, (\theta_v-1)!  =
  \frac{\theta_v + 1}{2}\, \theta_v!   \leq \theta_v\, \theta_v! 
  \leq d\, \theta_v!.
\]
Finally, suppose that $v$ is the head of $\theta_v$ arcs of each colour.
Then $v$ may be good, or it may be the head of one bad pair of each colour,
or the head of two bad pairs of each colour.
The number of pairings of in-arcs around $v$ with two bad pairs of in-arcs
of each colour is
\[ 9\, \binom{\theta_v}{4}^2 \, (\theta_v - 4)!
 = 
    \frac{\theta_v(\theta_v-1)(\theta_v-2)(\theta_v-3)}{64}\, \theta_v!
 \leq \theta_v^4\cdot \theta_v!\\
  \leq d^4\, \theta_v!,
\]
while the number of pairings of in-arcs around $v$ with one bad pair of
each colour is
\[ \binom{\theta_v}{2}^2\, (\theta_v-2)! 
 = \frac{\theta(\theta_v-1)}{4}\, \theta_v! \leq \theta_v^2\,\theta_v!
     \leq d^2\, \theta_v!.
\]
By symmetry, the same bounds hold for out-arcs and also hold after
exchanging green and yellow.
Since there are at most 16 bad pairs, it follows that
\begin{equation}
 |\Psi'(H,L)|\leq d^{16}\, |\Psi(G,G')|.
\label{psi-prime}
\end{equation}
Now write $\mathbf{1}(e\in\gamma_\psi(G,G'))$ to denote the indicator
variable which is 1 if $e\in\gamma_\psi(G,G')$ and is 0 otherwise,
for  $(G,G')\in\Omega_{n,d}$ and $\psi\in\Psi(G,G')$.
Then
\begin{align*}
|\Omega_{n,d}|^2 f(e) 
   &=  \sum_{(G,G')}\,\, \sum_{\psi\in\Psi(G,G')}\,
       \mathbf{1}(e\in\gamma_\psi(G,G'))\, |\Psi(G,G')|^{-1}\\
   &\leq 4\, \sum_{L\in\mathcal{L}(Z)}\,\, \sum_{\psi\in\Psi'(H,L)}\,
          \mathbf{1}(e\in\gamma_\psi(G,G'))\, |\Psi(G,G')|^{-1}\\
   &\leq 4\, \sum_{L\in\mathcal{L}(Z)}\,\, \sum_{\psi\in\Psi'(H,L)}\,
	                 |\Psi(G,G')|^{-1}\\
               &\leq 4\, \sum_{L\in\mathcal{L}(Z)} \, d^{16}\\
	       &\leq 100\, d^{22}\, n^6\, |\Omega_{n,d}|.
\end{align*}
The first inequality follows by Lemma~\ref{notquiteunique},
the third inequality follows from (\ref{psi-prime}),
and applying Lemma~\ref{poly} gives the last inequality. 
This completes the proof.
\end{proof}

We can now complete our argument by proving 
Proposition~\ref{our-second-largest-eigval}.

\begin{proof}[Proof of Proposition~\ref{our-second-largest-eigval}]\
For any transition $e=(Z,Z')$ of the switch chain
we have
\[ 1/Q(e) =  |\Omega_{n,d}|/P(Z,Z') =  \binom{dn}{2}\, |\Omega_{n,d}|.\]
Therefore, by Lemma~\ref{load}, 
\begin{equation} 
\label{rho-bound}
\rho(f) \leq 50 d^{24}\, n^8.
\end{equation}
Next, observe that $\ell(f) \leq dn$, since each step along a canonical path
replaces at least one arc of $G$ by an arc of $G'$.
The result follows from Lemma~\ref{second-eigval}. 
\end{proof}

\section{An illustrative example}\label{a:example}

Let $(G,G')\in\Omega_{n,d}$ be any pair of digraphs 
with the symmetric difference $H$ given in
Figure~\ref{example1}, where  vertices of degree 0 in $H$ are not shown.
To avoid congestion in the figure, some vertices are depicted as black
rectangles.
Solid arcs belong to $G$ and dashed arcs belong to $G'$, so they play
the role of blue and red arcs.
\begin{center}
\begin{figure}[ht]
 \psfrag{v}{$v$}\psfrag{x00}{$x_{0,0}$} \psfrag{x01}{$x_{0,1}$}\psfrag{x11}{$x_{1,1}$}
 \psfrag{x10}{$x_{1,0}$} \psfrag{z00}{$z_{0,0}$} \psfrag{z01}{$z_{0,1}$}
 \psfrag{z11}{$z_{1,1}$} \psfrag{z10}{$z_{1,0}$}
 \psfrag{u1}{$u_1$} \psfrag{u2}{$u_2$} \psfrag{w1}{$w_1$} \psfrag{w2}{$w_2$}
 \psfrag{r1}{$r_1$} \psfrag{r2}{$r_2$} \psfrag{t1}{$t_1$} \psfrag{t2}{$t_2$}
 \psfrag{p1}{$p_1$} \psfrag{p2}{$p_2$} \psfrag{q1}{$q_1$} \psfrag{q2}{$q_2$}
 \psfrag{s1}{$s_1$} \psfrag{s2}{$s_2$} 
\centerline{\includegraphics[scale=0.48]{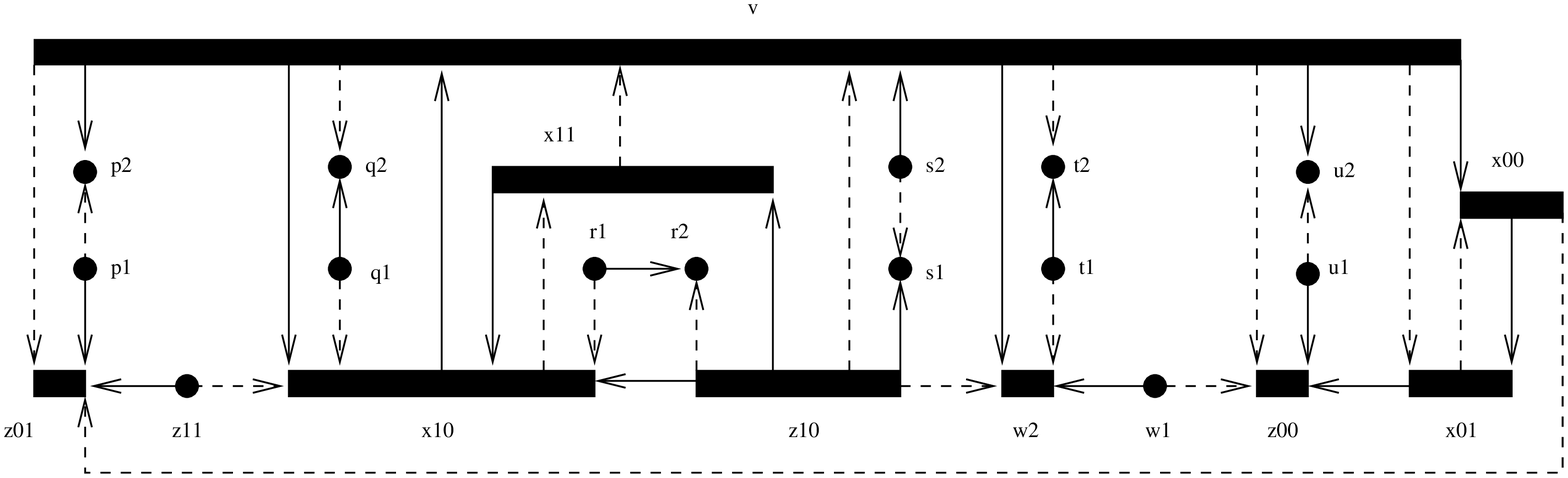}}
\caption{The symmetric difference $H$ of $G$ and $G'$.}
\label{example1}
\end{figure}
\end{center}
\vspace*{-\baselineskip}
Let $\psi$ which be the pairing which produces the forward circuits
\begin{equation} 
\label{circuits}
\begin{split}
& v x_{0,0} x_{0,1} z_{0,0} w_1 w_2 z_{1,0} x_{1,1} x_{1,0} v x_{1,1} x_{1,0} z_{1,1} z_{0,1} x_{0,0} x_{0,1},
\quad v p_2 p_1 z_{0,1},\quad 
v x_{1,0} q_1 q_2, \\
& \hspace*{2cm} z_{1,0} x_{1,0} r_1 r_2,\quad 
z_{1,0}  v s_2 s_1,\quad
v w_2 t_1 t_2,\quad 
v u_2 u_1 z_{0,0}, 
\end{split}
\end{equation}
in the given order.
Set $Z_0=G$ and start processing $H$.  The first circuit to process is
the eccentric 2-circuit
\[ S = 
 v x_{0,0} x_{0,1} z_{0,0} w_1 w_2 z_{1,0} x_{1,1} x_{1,0} v x_{1,1} x_{1,0} z_{1,1} z_{0,1} x_{0,0} x_{0,1}.
\]
We have $(i,h)=(0,0)$,  and the eccentric arc $(z_{1,0},v)$
does not belong to $A(S)$.  Hence $S$ falls into case (Ea) and 
we must first perform the eccentric switch $[z_{1,0} x_{1,1} x_{1,0} v]$.
This produces the next digraph $Z_1$ in the canonical path $\gamma_\psi(G,G')$.
The eccentric arc $(z_{1,0},v)$ has been used in the eccentric switch,
so it is now an interesting arc.  
Initially it belonged to $G'-G$,
and now it belongs to $Z_1\cap G'$, so it does not belong to the
current symmetric difference $Z_1\triangle G'$.  However, we include
all interesting arcs in our figures, denoted by thicker arcs
(either solid or broken, as appropriate).
Hence Figure~\ref{example2} shows the symmetric difference of $Z_1$ and $G'$,
together with the eccentric arc. 
\begin{center}
\begin{figure}[ht]
 \psfrag{v}{$v$}\psfrag{x00}{$x_{0,0}$} \psfrag{x01}{$x_{0,1}$}\psfrag{x11}{$x_{1,1}$}
 \psfrag{x10}{$x_{1,0}$} \psfrag{z00}{$z_{0,0}$} \psfrag{z01}{$z_{0,1}$}
 \psfrag{z11}{$z_{1,1}$} \psfrag{z10}{$z_{1,0}$}
 \psfrag{u1}{$u_1$} \psfrag{u2}{$u_2$} \psfrag{w1}{$w_1$} \psfrag{w2}{$w_2$}
 \psfrag{r1}{$r_1$} \psfrag{r2}{$r_2$} \psfrag{t1}{$t_1$} \psfrag{t2}{$t_2$}
 \psfrag{p1}{$p_1$} \psfrag{p2}{$p_2$} \psfrag{q1}{$q_1$} \psfrag{q2}{$q_2$}
 \psfrag{s1}{$s_1$} \psfrag{s2}{$s_2$} 
\centerline{\includegraphics[scale=0.48]{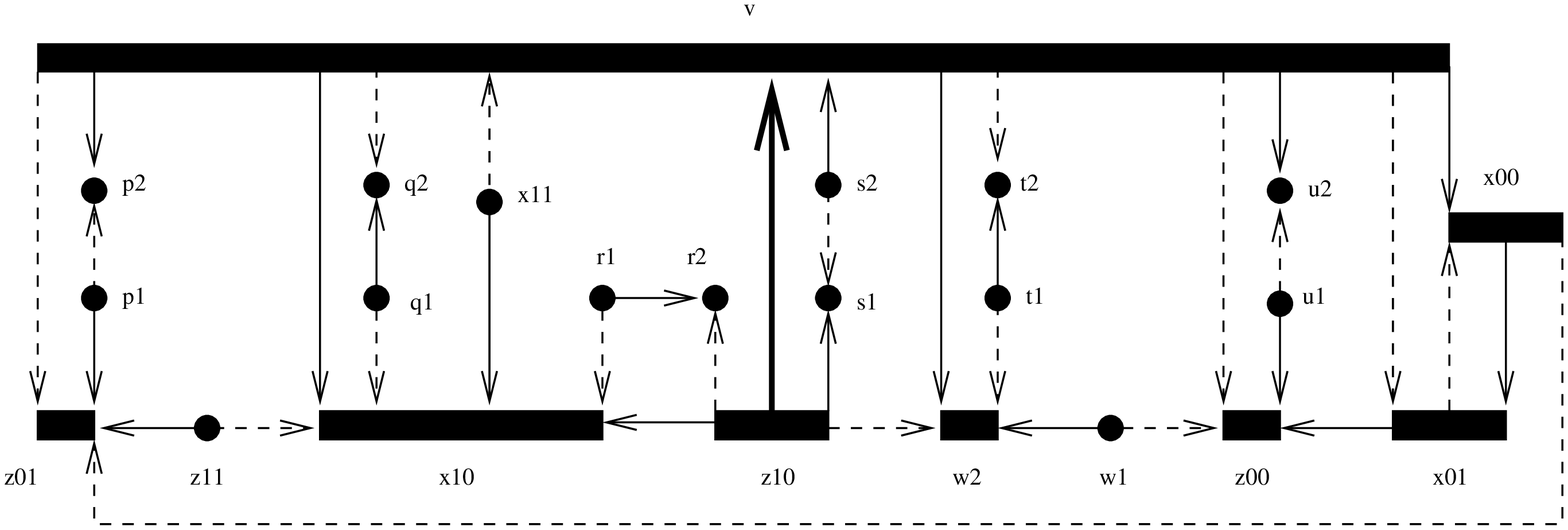}}
\caption{The symmetric difference of $Z_1$ and $G'$, together with the
single interesting arc (the eccentric arc), after the eccentric switch.}
\label{example2}
\end{figure}
\end{center}
The arcs $(z_{1,0},x_{1,1}), (x_{1,0},x_{1,1}), (x_{1,0},v)$ have 
disappeared because they have now been switched to agree with $G'$.
They play no further part in the formation of the canonical path.

Next, we must process the normal 2-circuit
\[ S' = 
 v x_{0,0} x_{0,1} z_{0,0} w_1 w_2 z_{1,0} v x_{1,1} x_{1,0} z_{1,1} z_{0,1} x_{0,0} x_{0,1}
\]
which the eccentric switch has produced (see Figure~\ref{eccentric-detail}).
From Lemma~\ref{eccentric-plus-shortcut} we know that the shortcut arc is 
$(z_{1,0},x_{1,0})$, and again $(i,h)=(0,0)$.
Now $(z_{1,0},x_{1,0})\in A(Z_1)$ so $S'$ falls into case (Nc), and we
will perform the shortcut switch last.  Our next task is to process the
1-circuit
\[ S_1 = 
 v x_{0,0} x_{0,1} z_{0,0} w_1 w_2 z_{1,0} x_{1,0} z_{1,1} z_{0,1} 
                x_{0,0} x_{0,1}.
\]
The set $\mathcal{B}$ of end-vertices of odd chords which are absent in $Z_1$
is $\mathcal{B} = \{ x_{0,1}\, z_{0,1},\, z_{0,0}\}$.
Now $z_{0,1}x_{0,0}x_{0,1}$ is a contiguous substring of $S$, so these 
vertices are all distinct, and hence $\mathcal{B}$ has three elements. 
Thus there will be three phases in the processing of $S_1$.
The first phase is over after just one switch, namely 
$[v x_{0,0} x_{0,1} z_{0,0}]$.
This produces the next digraph
$Z_2$ on the canonical path: see Figure~\ref{example3}.
\begin{center}
\begin{figure}[ht]
 \psfrag{v}{$v$}\psfrag{x00}{$x_{0,0}$} \psfrag{x01}{$x_{0,1}$}\psfrag{x11}{$x_{1,1}$}
 \psfrag{x10}{$x_{1,0}$} \psfrag{z00}{$z_{0,0}$} \psfrag{z01}{$z_{0,1}$}
 \psfrag{z11}{$z_{1,1}$} \psfrag{z10}{$z_{1,0}$}
 \psfrag{u1}{$u_1$} \psfrag{u2}{$u_2$} \psfrag{w1}{$w_1$} \psfrag{w2}{$w_2$}
 \psfrag{r1}{$r_1$} \psfrag{r2}{$r_2$} \psfrag{t1}{$t_1$} \psfrag{t2}{$t_2$}
 \psfrag{p1}{$p_1$} \psfrag{p2}{$p_2$} \psfrag{q1}{$q_1$} \psfrag{q2}{$q_2$}
 \psfrag{s1}{$s_1$} \psfrag{s2}{$s_2$} 
\centerline{\includegraphics[scale=0.5]{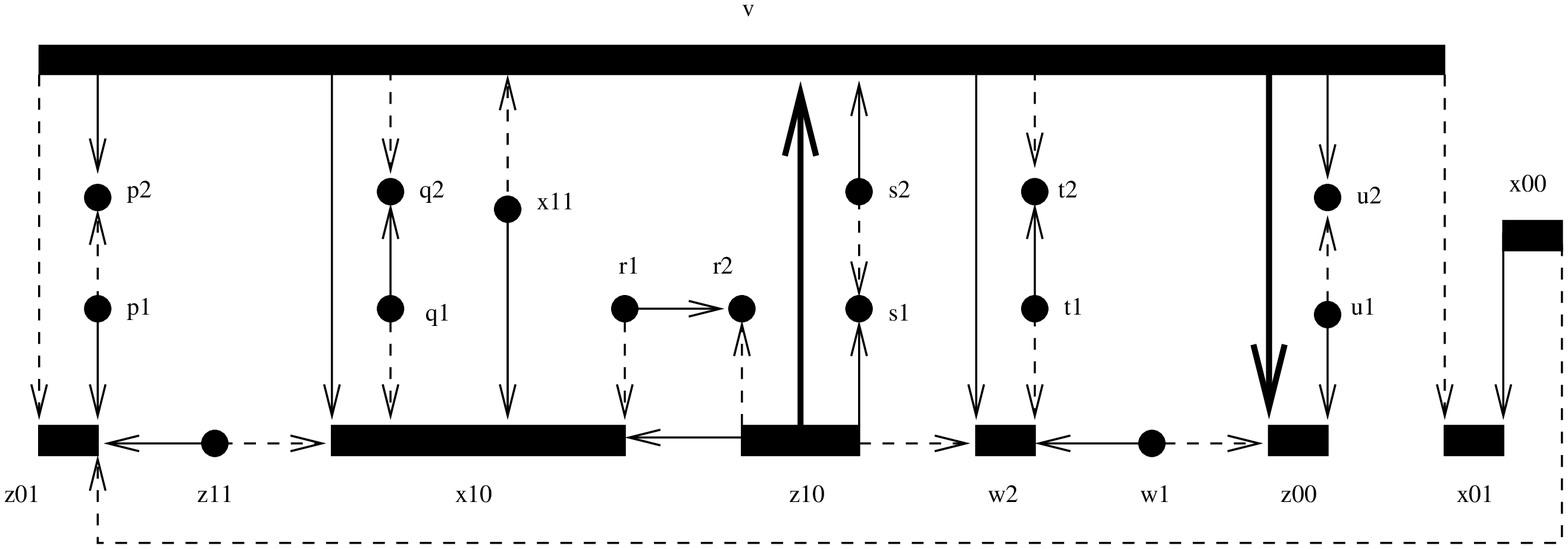}}
\caption{The symmetric difference of $Z_2$ and $G'$,
together with the two interesting arcs (the eccentric arc and one odd chord),
after Phase 1.}
\label{example3}
\end{figure}
\end{center}
\vspace*{-2\baselineskip}
The odd chord $(v,z_{0,0})$ has become an interesting arc, so it is included
in Figure~\ref{example3} together with the eccentric arc.  Both belong to
$Z_2\cap G'$, and hence they are depicted by a thick unbroken arc.
The arcs $(v,x_{0,0})$, $(x_{0,1},x_{0,0})$, $(x_{0,1},z_{0,0})$ have now
been switched to agree with $G'$, so they play no further role.  Hence
we have omitted these arcs from Figure~\ref{example3}. 

We now start Phase 2 with the switch $[v x_{1,0} z_{1,1} z_{0,1}]$,
producing the next digraph $Z_3$ on the canonical path.  
See Figure~\ref{example4}.
Note that there are four interesting arcs in $Z_3$, namely three odd chords
and the eccentric arc. The vertex $z_{1,1}$ is omitted
from Figure~\ref{example4} since it has degree zero in the symmetric
difference of $Z_3$ and $G'$.
(We will make no further comments on the inclusion of interesting arcs or the
omission of isolated vertices for the remaining figures.)
\begin{center}
\begin{figure}[ht]
 \psfrag{v}{$v$}\psfrag{x00}{$x_{0,0}$} \psfrag{x01}{$x_{0,1}$}\psfrag{x11}{$x_{1,1}$}
 \psfrag{x10}{$x_{1,0}$} \psfrag{z00}{$z_{0,0}$} \psfrag{z01}{$z_{0,1}$}
 \psfrag{z11}{$z_{1,1}$} \psfrag{z10}{$z_{1,0}$}
 \psfrag{u1}{$u_1$} \psfrag{u2}{$u_2$} \psfrag{w1}{$w_1$} \psfrag{w2}{$w_2$}
 \psfrag{r1}{$r_1$} \psfrag{r2}{$r_2$} \psfrag{t1}{$t_1$} \psfrag{t2}{$t_2$}
 \psfrag{p1}{$p_1$} \psfrag{p2}{$p_2$} \psfrag{q1}{$q_1$} \psfrag{q2}{$q_2$}
 \psfrag{s1}{$s_1$} \psfrag{s2}{$s_2$} 
\centerline{\includegraphics[scale=0.5]{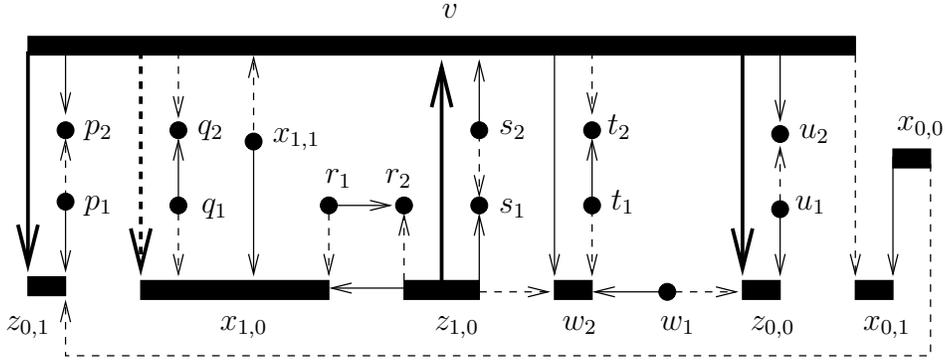}}
\caption{The symmetric difference of $Z_3$ and $G'$,
together with the four interesting arcs (the eccentric arc and three 
odd chords), after the first step of Phase 2.}
\label{example4}
\end{figure}
\end{center}
\vspace*{-2\baselineskip}
The next step in Phase 2 is the switch $[v w_2 z_{1,0} x_{1,0}]$,
which involves the shortcut arc.  This produces the digraph $Z_4$ on
the canonical path.  See Figure~\ref{example5}.
Note that $Z_4$ has five interesting arcs, namely 
\[ (z_{1,0},v),\,\, (z_{1,0},x_{1,0}),\,\, (v,z_{0,1}),\,\, (v,w_2),\,\, 
(v,z_{0,0}).
\]
This is the maximum possible, by Lemma~\ref{zoo}.    
Later we will show that $Z_4$ also has the maximum number of bad
pairs. 
\begin{center}
\begin{figure}[ht]
 \psfrag{v}{$v$}\psfrag{x00}{$x_{0,0}$} \psfrag{x01}{$x_{0,1}$}\psfrag{x11}{$x_{1,1}$}
 \psfrag{x10}{$x_{1,0}$} \psfrag{z00}{$z_{0,0}$} \psfrag{z01}{$z_{0,1}$}
 \psfrag{z11}{$z_{1,1}$} \psfrag{z10}{$z_{1,0}$}
 \psfrag{u1}{$u_1$} \psfrag{u2}{$u_2$} \psfrag{w1}{$w_1$} \psfrag{w2}{$w_2$}
 \psfrag{r1}{$r_1$} \psfrag{r2}{$r_2$} \psfrag{t1}{$t_1$} \psfrag{t2}{$t_2$}
 \psfrag{p1}{$p_1$} \psfrag{p2}{$p_2$} \psfrag{q1}{$q_1$} \psfrag{q2}{$q_2$}
 \psfrag{s1}{$s_1$} \psfrag{s2}{$s_2$} 
\centerline{\includegraphics[scale=0.5]{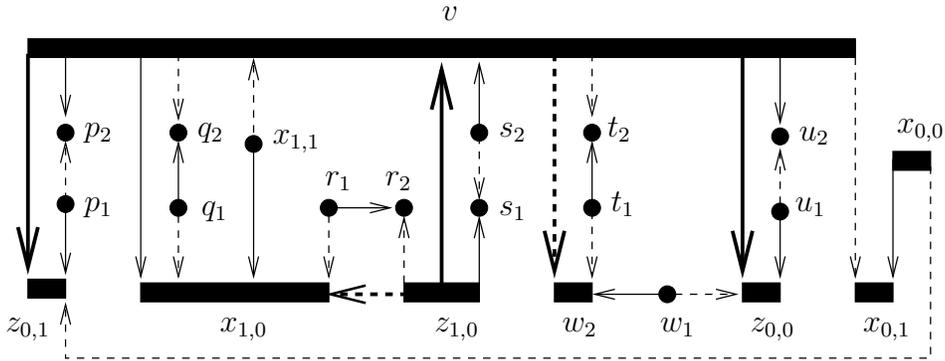}}
\caption{The symmetric difference of $Z_4$ and $G'$ 
together with the five interesting arcs (the eccentric arc, the shortcut arc 
and three odd chords), after the second step of Phase 2.}
\label{example5}
\end{figure}
\end{center}
\vspace*{-2\baselineskip}
The final step in Phase 2 is the switch $[vz_{0,0}w_1 w_2]$,
producing the digraph $Z_5$.
See Figure~\ref{example6}.  Now only one odd chord is interesting, as two
have been restored to their original state.
\begin{center}
\begin{figure}[ht]
 \psfrag{v}{$v$}\psfrag{x00}{$x_{0,0}$} \psfrag{x01}{$x_{0,1}$}\psfrag{x11}{$x_{1,1}$}
 \psfrag{x10}{$x_{1,0}$} \psfrag{z00}{$z_{0,0}$} \psfrag{z01}{$z_{0,1}$}
 \psfrag{z11}{$z_{1,1}$} \psfrag{z10}{$z_{1,0}$}
 \psfrag{u1}{$u_1$} \psfrag{u2}{$u_2$} \psfrag{w1}{$w_1$} \psfrag{w2}{$w_2$}
 \psfrag{r1}{$r_1$} \psfrag{r2}{$r_2$} \psfrag{t1}{$t_1$} \psfrag{t2}{$t_2$}
 \psfrag{p1}{$p_1$} \psfrag{p2}{$p_2$} \psfrag{q1}{$q_1$} \psfrag{q2}{$q_2$}
 \psfrag{s1}{$s_1$} \psfrag{s2}{$s_2$} 
\centerline{\includegraphics[scale=0.5]{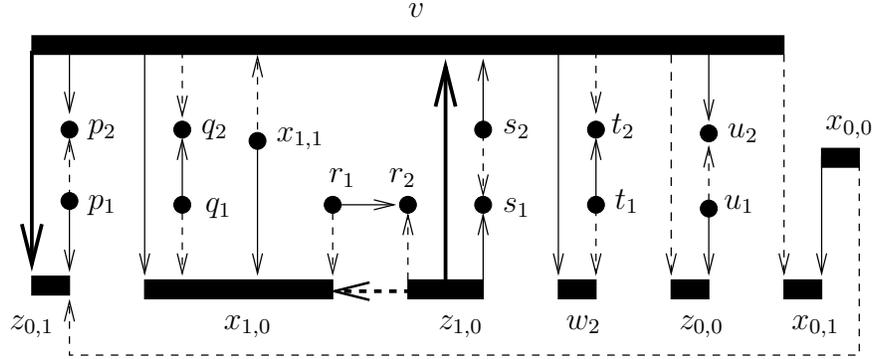}}
\caption{The symmetric difference of $Z_5$ and $G'$,
together with the three interesting arcs (the eccentric arc, the shortcut arc 
and one odd chord), after Phase 2.}
\label{example6}
\end{figure}
\end{center}
\vspace*{-2\baselineskip}
Then we perform Phase 3, which consists of one step: the switch
$[vz_{0,1}x_{0,0} x_{0,1}]$.  This produces the digraph $Z_6$ which has
no interesting odd chords, but still has two interesting arcs, namely the
 eccentric arc and shortcut arc.  See Figure~\ref{example7}.
\begin{center}
\begin{figure}[ht!]
 \psfrag{v}{$v$}\psfrag{x00}{$x_{0,0}$} \psfrag{x01}{$x_{0,1}$}\psfrag{x11}{$x_{1,1}$}
 \psfrag{x10}{$x_{1,0}$} \psfrag{z00}{$z_{0,0}$} \psfrag{z01}{$z_{0,1}$}
 \psfrag{z11}{$z_{1,1}$} \psfrag{z10}{$z_{1,0}$}
 \psfrag{u1}{$u_1$} \psfrag{u2}{$u_2$} \psfrag{w1}{$w_1$} \psfrag{w2}{$w_2$}
 \psfrag{r1}{$r_1$} \psfrag{r2}{$r_2$} \psfrag{t1}{$t_1$} \psfrag{t2}{$t_2$}
 \psfrag{p1}{$p_1$} \psfrag{p2}{$p_2$} \psfrag{q1}{$q_1$} \psfrag{q2}{$q_2$}
 \psfrag{s1}{$s_1$} \psfrag{s2}{$s_2$} 
\centerline{\includegraphics[scale=0.5]{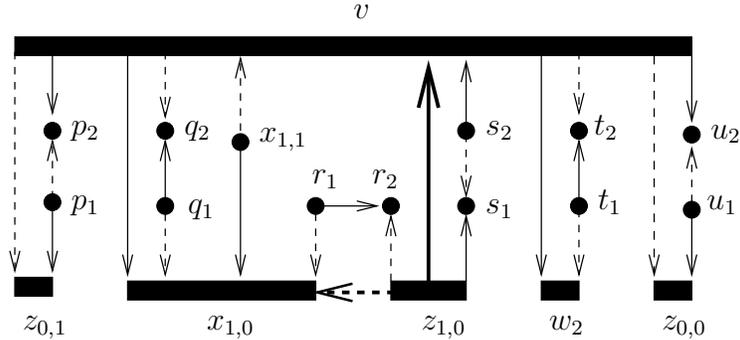}}
\caption{The symmetric difference of $Z_6$ and $G'$ together with
the two eccentric arcs (the eccentric arc and the shortcut arc),
after Phase 3: the processing of the 1-circuit is complete.}
\label{example7}
\end{figure}
\end{center}
\vspace*{-2\baselineskip}
This completes the processing of the 1-circuit $S_1$.  To complete the 
processing of the normal 2-circuit $S'$ we must perform the shortcut switch
$[x_{1,1} x_{1,0} z_{1,0} v]$.  This produces the digraph $Z_7$ as in
Figure~\ref{example8}, with no interesting arcs.
\begin{center}
\begin{figure}[ht!]
 \psfrag{v}{$v$}\psfrag{x00}{$x_{0,0}$} \psfrag{x01}{$x_{0,1}$}\psfrag{x11}{$x_{1,1}$}
 \psfrag{x10}{$x_{1,0}$} \psfrag{z00}{$z_{0,0}$} \psfrag{z01}{$z_{0,1}$}
 \psfrag{z11}{$z_{1,1}$} \psfrag{z10}{$z_{1,0}$}
 \psfrag{u1}{$u_1$} \psfrag{u2}{$u_2$} \psfrag{w1}{$w_1$} \psfrag{w2}{$w_2$}
 \psfrag{r1}{$r_1$} \psfrag{r2}{$r_2$} \psfrag{t1}{$t_1$} \psfrag{t2}{$t_2$}
 \psfrag{p1}{$p_1$} \psfrag{p2}{$p_2$} \psfrag{q1}{$q_1$} \psfrag{q2}{$q_2$}
 \psfrag{s1}{$s_1$} \psfrag{s2}{$s_2$} 
\centerline{\includegraphics[scale=0.5]{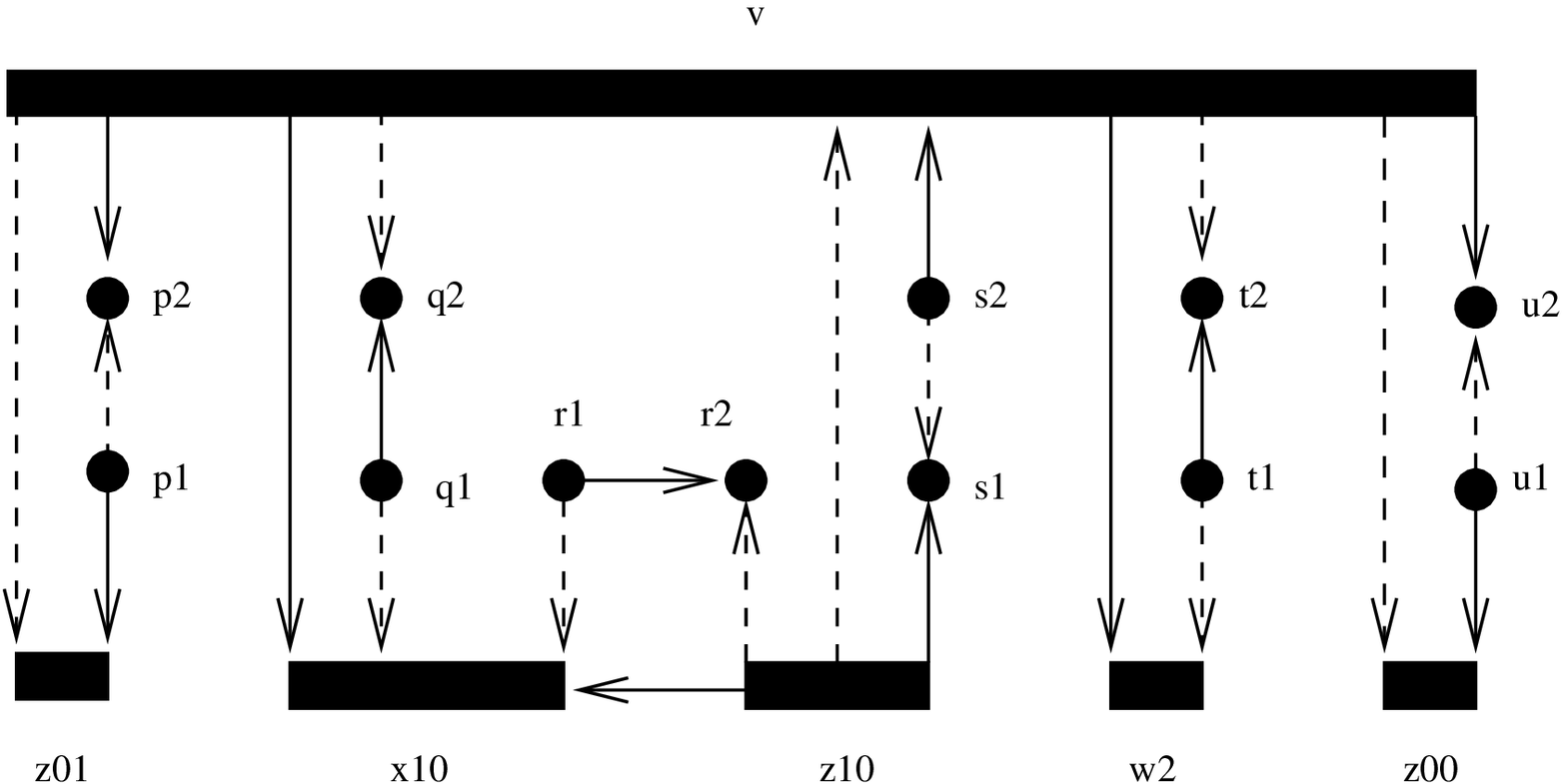}}
\caption{The symmetric difference of $Z_7$ and $G'$.}
\label{example8}
\end{figure}
\end{center}
This completes the processing of the normal 2-circuit $S'$, and hence it also
completes the processing of the eccentric 2-circuit $S$.

It remains to process the remaining circuits in the given order.
Each remaining circuit is an alternating 4-cycle, which
is processed by a single switch, removing it from the symmetric difference. 
This gives 6 more switches, specifically
\[ [v p_2 p_1 z_{0,1}], \,\, [v x_{1,0} q_1 q_2],\,\, 
  [z_{1,0} x_{1,0} r_1 r_2],\,\, 
   [z_{1,0} s_1 s_2 v],\,\,
  [v w_2 t_1 t_2],\,\, [v u_2 u_1 z_{0,0}].
\]
The switches are performed in this order,
producing digraphs $Z_8,\ldots, Z_{13}$ where $Z_{13}=G'$.  This
completes the construction of
the canonical path $\gamma_\psi(G,G')$ from $G$ to $G'$ corresponding to
$\psi$.  

\bigskip

Now let us return to the digraph $Z_4$.  We now show that there are
16 bad pairs in $Z_4$ with respect to $\psi$.
We redraw $H$
in Figure~\ref{example9}, where now solid lines show arcs in $H\cap Z_4$ 
and dashed lines show arcs in $H - Z_4$.   Hence solid and dashed
arcs play the role of green and yellow arcs, in the terminology of
Lemma~\ref{load}.
Interesting arcs are still shown with thicker lines.
\begin{center}
\begin{figure}[ht!]
 \psfrag{v}{$v$}\psfrag{x00}{$x_{0,0}$} \psfrag{x01}{$x_{0,1}$}\psfrag{x11}{$x_{1,1}$}
 \psfrag{x10}{$x_{1,0}$} \psfrag{z00}{$z_{0,0}$} \psfrag{z01}{$z_{0,1}$}
 \psfrag{z11}{$z_{1,1}$} \psfrag{z10}{$z_{1,0}$}
 \psfrag{u1}{$u_1$} \psfrag{u2}{$u_2$} \psfrag{w1}{$w_1$} \psfrag{w2}{$w_2$}
 \psfrag{r1}{$r_1$} \psfrag{r2}{$r_2$} \psfrag{t1}{$t_1$} \psfrag{t2}{$t_2$}
 \psfrag{p1}{$p_1$} \psfrag{p2}{$p_2$} \psfrag{q1}{$q_1$} \psfrag{q2}{$q_2$}
 \psfrag{s1}{$s_1$} \psfrag{s2}{$s_2$} 
\centerline{\includegraphics[scale=0.5]{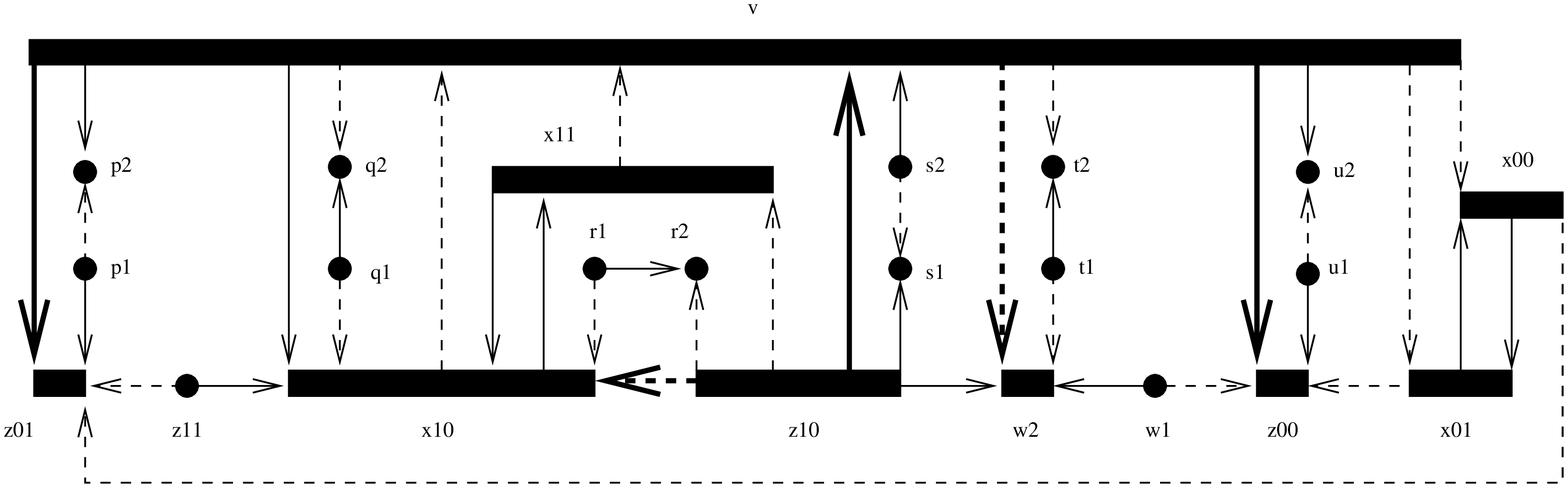}}
\caption{The symmetric difference $H$, where now solid lines are arcs in 
$H\cap Z_4$ and dashed lines are arcs in $H - Z_4$.}
\label{example9}
\end{figure}
\end{center}
\vspace*{-2\baselineskip}

By tracing around this figure using the circuits given in (\ref{circuits})
determined by the pairing $\psi$, we find that there are 16 bad pairs in
$Z_4$ with respect to $\psi$.  This is the maximum possible number
of bad pairs, as proved in Lemma~\ref{load}.  Table~\ref{badpairs}
shows the bad vertices and the bad pairs of arcs incident with each one.
\begin{table}[ht]
\renewcommand{\arraystretch}{1.5}
\begin{center}
\begin{tabular}{|r|l|l|}
\hline
bad vertex & bad green pairs & bad yellow pairs \\
\hline
$v$ & $(z_{1,0},v),\  (s_2,v)$ & $(x_{1,1},v),\  (x_{1,0},v)$ \\
     & $(v,z_{0,0}),\  (v,u_2)$ & $(v,x_{0,0}),\  (v,x_{0,1})$ \\
     &  $(v,z_{0,1}),\  (v,p_2)$ & $(v,w_2),\  (v,t_2)$\\
\hline
$z_{1,0}$ & $(z_{1,0},v), \ (z_{1,0},s_1)$ &  
                 $(z_{1,0},x_{1,0}),\ (z_{1,0},r_2)$ \\
\hline
$x_{1,0}$ &  $(x_{1,1},x_{1,0}),\ (z_{1,1},x_{1,0})$ & 
             $(z_{1,0},x_{1,0}),\ (r_1,x_{1,0})$   \\
\hline
$z_{0,0}$ &  $(v,z_{0,0}),\ (u_1,z_{0,0})$ & 
             $(x_{0,1},z_{0,0}),\ (w_1,z_{0,0})$   \\
\hline
$z_{0,1}$ &  $(v,z_{0,1}),\ (p_1,z_{0,1})$ 
          &  $(z_{1,1},z_{0,1}),\ (x_{0,0},z_{0,1})$   \\
\hline
$w_2$     &  $(w_1,w_2),\ (z_{1,0},w_2)$ &  
             $(v,w_2),\ (t_1,w_2)$   \\
\hline
\end{tabular}
\caption{The bad vertices and bad pairs of arcs in $Z_4$ with respect to
$\psi$.}
\label{badpairs}
\end{center}
\end{table}

We now make two final comments.
\begin{enumerate}
\item[(i)] In this relatively small example, not many coincidences between the
bad vertices are possible.  For instance, we know that $w_2\neq z_{0,0}$ since 
$z_{0,0}w_1w_2$ is a contiguous substring of $S$, while $w_2\neq x_{1,0}$ since
$(v,w_2)$ is a blue arc in $H$ and $(v,x_{1,0})$ is a red arc in $H$.
In our example, the only coincidences that may occur are that $z_{1,0}$ may equal 
$z_{0,0}$
or it may equal $z_{0,1}$.  If either holds then 
the vertex $z_{0,0}$ is incident with four bad pairs in $Z_4$, one of each colour and
orientation.  
\item[(ii)] This example was constructed to produce a digraph with the maximum
number of bad pairs (namely $Z_4$, with 16 bad pairs).  This was
achieved by letting the interesting arcs all belong to $H$, so that they
did not become bad arcs when they became interesting, but instead they
created extra bad pairs with respect to $\psi$.
If instead $H$ just
consisted of the arcs of the eccentric 2-circuit $S$, then
any interesting arc would also be a bad arc.
(For example, if the eccentric arc had not been
an arc of $H$ but was absent in both $G$ and $G'$ then in $Z_1$ it would 
become a bad arc with label $-1$.)
Then the analogue of $Z_4$ would be an example of a digraph 
with the maximum number of bad arcs. 
\end{enumerate}

\newpage

\end{document}